\documentclass[12pt]{amsart}

\usepackage{mathrsfs,amssymb,amsmath}
\usepackage{verbatim}
\usepackage{hyperref}
\usepackage{dsfont}

\newif\ifanswers
\answerstrue

\usepackage{amsmath,amssymb}
\usepackage{graphicx}

\newtheorem{theorem}{Theorem}[section]
\newtheorem{proposition}[theorem]{Proposition}
\newtheorem{lemma}[theorem]{Lemma}
\newtheorem{corollary}[theorem]{Corollary}
\theoremstyle{definition}
\newtheorem{definition}[theorem]{Definition}

\newtheorem{remark}[theorem]{Remark}

\newcommand{\vol}{\operatorname{Vol}}

\newcommand{\Z}{{\mathbb Z}}

\newcommand{\R}{{\mathbb R}}
\newcommand{\T}{{\mathbb T}}

\newcommand{\M}{{\mathcal{M}}}

\newcommand{\Ac}{\mathcal{A}}

\newcommand{\Gc}{\mathcal{G}}

\newcommand{\Xc}{\mathcal{X}}

\newcommand{\CU}{\mathcal{U}}

\newcommand{\Vc}{\mathcal{V}}
\newcommand{\Ec}{\mathcal{E}}
\newcommand{\E}{\mathbb{E}}

\title[Planck-scale mass equidistribution]{Planck-scale mass equidistribution of toral Laplace eigenfunctions}

\author[A. Granville]{Andrew Granville}
\address{AG: D\'epartement de math\'ematiques et de statistique\\
Universit\'e de Montr\'eal\\
CP 6128 succ. Centre-Ville\\
Montr\'eal, QC H3C 3J7\\
Canada;
and Department of Mathematics \\
University College London \\
Gower Street \\
London WC1E 6BT \\
England.
}
\email{{\tt and.granville@gmail.com}}

\author[I. Wigman]{Igor Wigman}
\address{IW: Department of Mathematics \\ King's College London \\ Strand \\ London WC2R 2LS \\ England, UK   }
\email{{\tt igor.wigman@kcl.ac.uk}}


\begin{document}

\begin{abstract}
We study the small scale distribution of the $L^2$-mass of
eigenfunctions of the Laplacian on the the two-dimensional flat torus. Given an orthonormal
basis of eigenfunctions, Lester and Rudnick \cite{LeRu}  showed the existence of a density one subsequence
whose $L^2$-mass  equidistributes more-or-less down to the Planck scale. We give a more precise version of their result showing equidistribution holds down to a small power of log above Planck scale, and also showing that the $L^2$-mass  fails to equidistribute at a slightly smaller power of log above the Planck scale.

This article rests on a number of results about the proximity of lattice points on circles, much of it based on foundational work of Javier Cilleruelo.
\end{abstract}

\date{\today}
\dedicatory{Dedicated to the memory of Javier Cilleruelo}

\maketitle

\section{Introduction}

\subsection{Background and motivation}
Let $\M$ be a  smooth, compact, $n$-dimensional Riemann manifold, and with no loss of generality we assume that $\vol(\M) = 1$.
We are interested in the Laplace spectrum of $\M$ (also called ``energy levels"): these are the eigenvalues, $E$, of the equation
\begin{equation*}
\Delta f + Ef=0.
\end{equation*}
It is well-known that the eigenvalue spectrum $\{E_{j}\}_{j\geq 1}$ is discrete, that $E_{j}\rightarrow\infty$, and we let $\phi_{j}$ be a corresponding
orthonormal basis of eigenfunctions. Shnirelman's Theorem \cite{Sn, Ze, CdV} asserts that if $\M$ is chaotic (that is, the geodesic
flow on $\M$ is ergodic),
then there is a subsequence $\{E_{j_{k}}\}_{j\geq k}$ of $\{ E_j\}_{j\geq 1}$, of density one,  for which the $\phi_{j}$ are
$L^{2}$-equidistributed in the phase-space; in particular for every ``nice'' domain on the configuration space $\Ac\subseteq  \M$ we have
\begin{equation}
\label{eq:L2 mass A->|A|}
\frac{\int\limits_{\Ac}\phi_{j_{k}}^{2}(y)dy}{\vol(\Ac)} \rightarrow 1,
\end{equation}
where $\vol(\Ac)$ is the ($n$-dimensional) volume of $\Ac$.

Berry's widely believed conjecture ~\cite{Berry1,Berry2} goes beyond Shnirelman's Theorem, asserting that \eqref{eq:L2 mass A->|A|}
holds for any $\Ac=\Ac_{j_{k}}$ which shrinks slightly slower than the {\em Planck scale}
$E_{j_{k}}^{-1/2}$. More precisely, let $$B_{x}(r)\subseteq\M$$ be the radius $r>0$ geodesic ball centred at $x$. Then there should exist
a density $1$ subsequence $\{ \phi_{j_k}\}_{k\geq 1}$ of energy levels, such that
\eqref{eq:L2 mass A->|A|}
holds uniformly for all $x\in \M$, $r>r_{0}(E_{j_{k}})$ with $B_{x}(r)$ in place of $A$, as long as
\begin{equation}
\label{eq:r0>>E^1/2}
\lim\limits_{E\rightarrow\infty} r_{0}(E) \cdot E^{1/2} =\infty,
\end{equation}
i.e.
\begin{equation}
\label{eq:unif equidist sup r>r0}
\sup\limits_{r>r_{0}(E), x\in\M} \left| \frac{\int\limits_{B_{x}(r)}\phi_{j_{k}}^{2}(y)dy}{\vol(B_{x}(r))}  - 1 \right|  \rightarrow 0.
\end{equation}
There are only a few such results in the literature with $r$ small: \

---\ Luo and Sarnak \cite{Luo-Sarnak} showed this for $r>E^{-\alpha}$ for some small $\alpha>0$, for the modular surface, where the eigenfunctions are the eigenfunctions of all Hecke operators, and Young \cite{Young}
showed this for all eigenfunctions for  $r>E^{-1/4+o(1)}$, assuming the Generalized Riemann Hypothesis;

---\ Hezari-Rivi\`{e}re ~\cite{Hezari-Riviere} and Han ~\cite{Han1} showed the integral is the expected value up to a multiplicative constant,  for $r>(\log E)^{-\alpha}$ for some small $\alpha>0$, on negatively curved manifolds. Han ~\cite{Han2} also showed this for ``symmetric" manifolds
(i.e.~ manifolds on which the group of isometries act transitively) on which the lower bound on $r$ depends on the growth rate of the eigenspace dimensions (the ``spectral degeneracy'').

\subsection{Toral eigenfunctions}
\label{sec:stat main res}

Our starting point is the work of Lester and Rudnick \cite[Theorem $1.1$]{LeRu}, who considered the small-scale
equidistribution \eqref{eq:unif equidist sup r>r0} of Laplace eigenfunctions
on the $d$-dimensional  torus $\T^{d}=(\R /\Z)^{d}$.
For $d=2$ they proved that if $\{\phi_{j}\}$ is an orthonormal basis of $L^{2}(\T^{2})$, then there exists a {\em density one}
subsequence $\{j_{k}\}$ of the positive integers, for which \eqref{eq:unif equidist sup r>r0} holds provided that $r_{0}>E^{-1/2+o(1)}$,
which is close to the full  (optimal) ``Planck range'' \eqref{eq:r0>>E^1/2}; we prove a strong version of this result below (see
Theorem \ref{thm:unif distr BR}).

Let $$S=\{a^{2}+b^{2}:\:a,b\in\Z\}$$ be the set of all integers expressible as sum of two squares.
For $n\in S$ let $\Ec_{n}$  be the set of  lattice points lying on the circle of radius $\sqrt{n}$, namely
\begin{equation*}
\Ec_{n} = \{ \lambda\in\Z^{2}:\: \|\lambda\|^{2}=n\},
\end{equation*}
which has size
$  \#\Ec_{n}=r_{2}(n)$, the number of different ways of expressing $n$ as the sum of two squares.

The eigenvalues of the Laplacian of $\T^{2}$ are the numbers $E=4\pi^{2}n$ with\footnote{By
an abuse of notation, $n$ is commonly referred to as an ``energy level" rather than the corresponding $E=4\pi^{2}n$.} 
$n\in S$, and the corresponding space of (complex-valued) eigenfunctions is
\begin{equation}
\label{eq:f(x)=sum c*exp}
f_{n}(x) = \sum\limits_{\lambda\in \Ec_{n}} c_{\lambda} e\left(\left\langle x,\lambda\right\rangle\right)
\end{equation}
of dimension $r_2(n)$. We will further assume that the $f_{n}$ are real-valued, so that
\begin{equation}
\label{eq:clambda,c-lambda}
c_{-\lambda}=\overline{c_{\lambda}},
\end{equation}
and multiply through by a constant so that
\begin{equation}
\label{eq:sum clamsqr=1}
\|f_{n}\|_{2}^{2} = \sum\limits_{\lambda\in\Ec_{n}}\|c_{\lambda}\|^{2} = 1.
\end{equation}
Landau ~\cite[\S1.8]{Landau,Brudern} proved that
\begin{equation}
\label{eq:Landau M/sqrt(logM)}
|\{ n\in S:\: n\le N  \} | \sim \kappa_{LR}\cdot\frac{N}{\sqrt{\log{N}}}\left(1+O\left(\frac{1}{\log{N}}\right)\right),
\end{equation}
where $\kappa_{LR}:=\frac \pi 4 \kappa'=0.76422\ldots$ with
\begin{equation}
\label{eq:kappa' def}
\kappa':= \prod_{p\equiv 1 \pmod 4} (1-1/p^2)^{1/2}.
\end{equation}
Ramanujan rediscovered this, and observed that
\begin{equation}
\label{eq:N=n^o(1)}
 r_2(n)=O_\epsilon(n^{\epsilon}) \ \  \text{ for every} \ \epsilon>0 .
\end{equation}

\subsection{Planck-scale mass equidistribution for Bourgain-Rudnick sequences}

\begin{definition}
For $\delta>0$ we say that a sequence $\{n\}\subseteq S$ satisfies the Bourgain-Rudnick condition, denoted by
$BR(\delta)$, if
there exists $C>0$ such that
\begin{equation}
\label{eq:BR(delta)}
\min\limits_{\lambda\ne \lambda'} \|\lambda-\lambda'\| > C\cdot n^{1/2-\delta}.
\end{equation}
\end{definition}

Bourgain-Rudnick ~\cite[Lemma 5]{Bourgain-Rudnick}
proved that for every $C>0$,
\begin{equation}
\label{eq:BR bound for BR seq}
\begin{split}
B(N;\delta)&:=\# \left\{n\le N,\ n\in S:\: \min\limits_{\lambda\ne \lambda'} \|\lambda-\lambda'\| \leq C\cdot n^{1/2-\delta}\right\}
\ll N^{1-\delta/3}.
\end{split}
\end{equation}
(The $B(N;\delta)$ is also implicitly dependent on $C$.)
Together with Landau's estimate \eqref{eq:Landau M/sqrt(logM)}, this implies that a generic sequence $\{n\}\subseteq S$ satisfies the $BR(\delta)$ condition, for arbitrary $\delta>0$.
Theorem \ref{thm:G* asymp} below allows us to improve this bound to $$B(N;\delta) \ll N^{1-\delta} (\log N)^{1/2},$$ which is perhaps close to the
true number of exceptional $n$. For all eigenfunctions corresponding to energy levels satisfying the $BR(\delta)$ condition
we prove the following uniform equidistribution  result with a strong upper bound on the discrepancy, for  close to the full Planck range:

\begin{theorem}
\label{thm:unif distr BR}
Let $\epsilon>\delta>0$ and $0<\eta < \epsilon-\delta$. For all sufficiently large $n$ satisfying Bourgain-Rudnick's $BR(\delta)$
condition and all $f$ for which $ \|f\|=1$ we have
\begin{equation}
\label{eq:L2 disc < n^-3/2eta}
\sup\limits_{x\in\T,\, r>n^{-1/2+\epsilon}}\left| \frac{ \int\limits_{B_{x}(r)}f(y)^{2}dy}{\pi r^{2}} - 1 \right| \ll n^{-3\eta/2} .
\end{equation}
\end{theorem}

Theorem \ref{thm:unif distr BR} implies Lester and Rudnick's result \cite{LeRu}
for $2$-dimensional tori (mentioned at the beginning of section \ref{sec:stat main res}), as
 \eqref{eq:BR bound for BR seq} is so much smaller than \eqref{eq:Landau M/sqrt(logM)}.

\subsection{On the number of exceptional energy levels}
Our goal is to estimate the number of exceptions to $BR(\delta)$, for given $\delta>0$.
To do this we obtain a precise estimate for
\begin{equation*}
B^*(N;\delta):=\#  \{(\lambda,\lambda'):\:  \|\lambda\|^{2}=\|\lambda'\|^{2} \le N, \, 0< \|\lambda-\lambda'\| \leq  C \|\lambda\|^{1-2\delta} \} ,
\end{equation*}
which yields a better bound than \eqref{eq:BR bound for BR seq} for the number of exceptions to $BR(\delta)$, as $B(N;\delta)\leq B^*(N;\delta)$.

\begin{theorem}
\label{thm:BR exceptions}
Fix $0<\delta<\frac 12$, and a constant $C>0$. Then
\begin{equation*}
B^*(N;\delta)=  \frac{4C}{\pi} \cdot \frac{1-2\delta}{1-\delta} \cdot N^{1-\delta}\log{N}\left(1+O\left( \frac{1}{\sqrt{\log{N}}} \right)\right) .
\end{equation*}
\end{theorem}

The proof of Theorem \ref{thm:BR exceptions}, given in section \ref{sec:ref aut}, is a relatively straightforward application
of a more general Theorem \ref{thm:not BR seq asympt Andrew}.
Theorem \ref{thm:BR exceptions} implies in particular the upper bound
\begin{equation}
\label{eq:BN upper bound mult}
B(N;\delta)\ll N^{1-\delta}\log{N}
\end{equation}
for the number of $n\in S$ {\em not} satisfying $BR(\delta)$, but not a lower bound. This is because Theorem \ref{thm:BR exceptions}
evaluates the number of close-by pairs of lattice points rather than the corresponding radii, which a priori can result in substantial over-counting in $B(N;\delta)$ as a particular radius might correspond to many different pairs.
That, in fact, this is so, follows from the following theorem; it implies that the average number of close-by pairs corresponding to 
radii not satisfying $BR(\delta)$ is growing to infinity (cf. \eqref{eq:BN upper bound single} vs. \eqref{eq:BN upper bound mult}).
We define
\begin{equation}
\label{eq:G^*def}
\Gc^*(N;M) := \{ n\le N:\: \exists\, \lambda,\lambda'\in \Ec_n. \, 0< \|\lambda-\lambda'\| < M\}.
\end{equation}

\begin{theorem}
\label{thm:G* asymp}

Let $M=M(N)$ be a function of $N$.

\begin{enumerate}

\item Under the assumption
\begin{equation}
\label{eq:assump M<=N^1/2/logN^17}
(\log N)^3\leq M\leq N^{1/2}/(\log N)^{17},
\end{equation}
$\Gc^*(N;M)$ satisfies the asymptotic law
\[
\# \Gc^*(N;M) =  \frac{\kappa'}{2} \cdot \sqrt{N}   M  ( (2\log M)^{1/2} +O(1)),
\]
where $\kappa'$ is as in \eqref{eq:kappa' def}.

\item With no assumption on $M$ we have the upper bound
\[
\# \Gc^*(N;M) \le  \frac{\kappa'}{2}\cdot \sqrt{N}   M  ( (2\log M)^{1/2} +O(1)).
\]
\end{enumerate}

\end{theorem}

The proof of Theorem \ref{thm:G* asymp}, given in section \ref{sec:NumAutoClasses},
is a straightforward application of the more general Theorem \ref{NumAutoClasses} below.
It also yields the aforementioned upper bound
\begin{equation}
\label{eq:BN upper bound single}
B(N;\delta)\ll  N^{1-\delta}(\log{N})^{1/2}
\end{equation}
for the number of $n\in S$ {\em not} satisfying $BR(\delta)$,
stronger than \eqref{eq:BN upper bound mult} above.
Comparing the second part of Theorem \ref{thm:G* asymp} to
\eqref{eq:Landau M/sqrt(logM)} we see that if $$\psi(n)=o(n^{1/2}/(\log n)),$$ then, for almost all $n$,
we have $ \|\lambda-\lambda'\| >\psi(n)$ whenever $$\|\lambda\|^{2}=\|\lambda'\|^{2}=n$$ with $\lambda\ne \lambda'$.
(Therefore $BR\left((1+\epsilon)\frac{\log\log n}{\log n}\right)$ holds for almost all $n$.)

\subsection{Planck-scale equidistribution for flat functions, valid for arbitrary energies}

Without ruling out the possible existence of {\em close-by pairs}
of lattice points we will not be able to prove a uniform result for all energy levels,
though we do get fairly precise results in terms of
\[
 \min_{  \lambda\ne \lambda' \in \Ec_n}  \|\lambda-\lambda'\| .
 \]
In Corollary \ref{cor:UpperBound} we show that Berry's conjecture is generically true for all
\[
r \geq   {r_2(n)^{2/3}(\log n)^\epsilon} /  { \min_{\lambda\ne \lambda'}  \|\lambda-\lambda'\| } ,
\]
while in Proposition \ref{prop:not equidist ce} we show that Berry's conjecture is generically false for some
\[ r \gg 1/{ \min_{\lambda\ne \lambda'}  \|\lambda-\lambda'\| } .\]
There is not much difference in these two bounds as  $r_2(n)$ is bounded by a small power of $\log n$, for almost all $n$.

We might instead ask for the typical error term, when considering the ball centre $x\in\T$ as random, uniformly distributed on the torus. By evaluating the variance
of the corresponding variable we will be able to infer that the $L^2$ mass is equidistributed for {\em most} $x$
(see Corollaries \ref{cor: Locally flat} and \ref{cor: Flat result} below).

For $f=f_{n}$ of the form \eqref{eq:f(x)=sum c*exp} and $r>0$ we define
$$X=X_{f,r}=X_{f,r;x} = \int\limits_{B_{x}(r)}f(y)^{2}dy,$$ thinking of $X$ as a {\em random variable}
with the ball centre $x\in\T$ drawn at random, {\em uniformly}, on the torus.
The expectation
\begin{equation*}
\E[X] = \int\limits_{\T}X_{f,r;x}dx,
\end{equation*}
of the $L^{2}$ mass is simply the area of $B_{x}(r)$, as $\|f\|=1$. Therefore:

\begin{lemma}
For every $r>0$ we have
\begin{equation*}
\E[X] =\pi r^{2}.
\end{equation*}
\end{lemma}

The corresponding variance is defined as
\begin{equation}
\label{eq:Var(X) def}
\Vc(X) =\int\limits_{\T}X_{f,r;x}^{2}dx-\E[X]^{2} =  \int\limits_{\T}\left(\int\limits_{B_{x}(r)}f(y)^{2}dy-\pi r^{2}\right)^{2}dx.
\end{equation}
The following result implies that the $L^{2}$-mass of any ``flat" $f$
is equidistributed on ``most" of the balls, i.e. $$\Vc(X)=o(r^{4}) = o(\E[X]^{2})$$
(see Corollaries \ref{cor: Locally flat} and \ref{cor: Flat result}).

\begin{theorem}
\label{thm:gen n var estimate}
For every $n\in S$, $f=f_{n}$ a function of the form \eqref{eq:f(x)=sum c*exp}, satisfying \eqref{eq:sum clamsqr=1},
and any small $\xi=\xi(n) > 0$, we have the following bound on the variance \eqref{eq:Var(X) def}
\begin{equation}
\label{eq:Var<< r^4}
\Vc(X) \ll \left( \sum\limits_{0 < \| \lambda-\lambda'\| < 1/(\xi r)} |c_{\lambda}c_{\lambda'}|^{2} + \xi^{3}  \right) \cdot r^{4},
\end{equation}
where the constant in the `$\ll$'-notation is absolute.
\end{theorem}

Note that if $\Vc(X) \ll   \delta^4 r^4$ then the measure of the set
$$
\Xc=\Xc(f_{n},r;\delta) =
\left\{ x\in \T: \: \left| {\int\limits_{B_{x}(r)}f(y)^{2}dy} - {\pi r^{2}} \right|>\delta r^2 \right\}
$$
of centres $x\in\T$ violating equidistribution, is $\ll \delta^2$.
The
following corollary shows that if the ``weights'' of the coefficients $c_\lambda$ are smoothly distributed around
the circle of radius $\sqrt{n}$, then $\Vc(X)=o(r^{4})$. To formulate it we will need the notation $\eta$ as follows.
Suppose that for some $\epsilon$ we have $|c_{\lambda}|^2\leq \epsilon$ for every $\lambda\in\Ec_{n}$.
Let
$$\eta=\eta(\{c_{\lambda}\}_{\lambda\in\Ec_{n}};\epsilon)>0$$ be the maximal possible number such that for all $\alpha\in \mathbb C$ with $|\alpha|=\sqrt{n}$ we have
\begin{equation}
\label{eq:Being Flat}
\sum_{\lambda:\ \| \lambda-\alpha\| < \eta \sqrt{n}} |c_{\lambda}|^2\leq \epsilon;
\end{equation}
that such $\eta$ exists follows from that fact that \eqref{eq:Being Flat} is satisfied for all $$\eta < \frac{1}{\sqrt{2}\sqrt{n}}.$$
One expects that for ``most" functions $f_{n}$ in \eqref{eq:f(x)=sum c*exp} satisfying \eqref{eq:sum clamsqr=1}, the inequality
\eqref{eq:Being Flat} is satisfied with $\eta\gg_\epsilon 1$ sufficiently small.


\begin{corollary}
\label{cor: Locally flat}
Fix $\epsilon >0$ and assume that for each $\lambda\in\Ec_{n}$ we have $|c_{\lambda}|^2\leq \epsilon$.
Then for all
$r\geq 1 / ( \epsilon \eta \sqrt{n})$ with $\eta$ as in \eqref{eq:Being Flat}, we have
$$\Vc(X) \ll \epsilon r^4,$$ where the constant in the `$\ll$'-notation is absolute.
\end{corollary}

That the assumptions of Corollary \ref{cor: Locally flat} are ``usually" satisfied (i.e. that the weight of the coefficients $c_\lambda$ are smoothly distributed for ``most" $f$) is supported by the following result which shows that the lattice points on the circle of radius $\sqrt{n}$ are not overly crowded together:

\begin{theorem}
\label{thm:close pairs}
Fix $\epsilon>0$ sufficiently small. For every integer $n\in S$ with $r_{2}(n)>0$ we have
\begin{equation}
\label{eq:close<n^1/2-eps<<r2^2-eps}
\#\left\{ \alpha,\beta \in \Ec_{n}:\: |\alpha-\beta| \le n^{1/2-\epsilon}\right\} \ll_{\epsilon} r_2(n)^{2-\epsilon}.
\end{equation}
\end{theorem}

One can show that if $$\epsilon_n:=\frac{ \log\log r_2(n)}{\log r_2(n)},$$ the bound is
\begin{equation}
\label{eq:close pairs eps=loglogr2/logr2}
\#\left\{ \alpha,\beta \in \Ec_{n}:\: |\alpha-\beta| \le n^{1/2-\epsilon_{n}}\right\} \ll \frac{r_2(n)^{2}}{\log r_2(n)}.
\end{equation}
The proof yields that the upper bound in Theorem \ref{thm:close pairs} may be improved to
\begin{equation*}
\#\left\{ \alpha,\beta \in \Ec_{n}:\: |\alpha-\beta| \le n^{1/2-\epsilon}\right\} \ll_{\epsilon} r_2(n)^{2-\tau \epsilon}
\end{equation*}
for any fixed $\tau<4$.

Theorem \ref{thm:close pairs} suggests that for all ``reasonable" choice of coefficients $c_{\lambda}$ the r.h.s. of \eqref{eq:Var<< r^4} is $o(r^{4})$. Here we propose a possible notion of ``reasonable''.

\begin{definition}[Flat and ultraflat functions]
\hfill \break

\begin{enumerate}
\item
Let $\{f_{n}\}_{n\in S}$ be a sequence of functions as in \eqref{eq:f(x)=sum c*exp}, and
\begin{equation}
\label{eq:epsn def}
\epsilon_n:=\frac{ \log\log r_2(n)}{\log r_2(n)}.
\end{equation}
We say that $\{f_{n}\}$ is {\em flat} if
\begin{equation}
\label{eq:  Flat f }
\max_{\substack{\alpha\in \mathbb C \\ |\alpha|=\sqrt{n}}} \ \ \
\sum_{\lambda\in \Ec_n:\ \| \lambda-\alpha\| < n^{1/2-\epsilon_n}} |c_{\lambda}|^2 = o_{n\to\infty}(1) .
\end{equation}

\item
Let $f=f_{n}$ be a function as in \eqref{eq:f(x)=sum c*exp}. For $\epsilon>0$ we say that $f$ is $\epsilon$-{\em ultraflat} if
for every $\lambda\in \Ec_n$,
\begin{equation}
\label{eq:clam ultraflat}
|c_\lambda|^2\leq \frac{1}{r_2(n)^{1-\epsilon}}.
\end{equation}

\end{enumerate}

\end{definition}

\begin{corollary}
\label{cor: Flat result}
\begin{enumerate}

\item
For all $\{f_{n}\}_{n\in S}$ flat with $r\gg 1/n^{1/2-2\epsilon_n}$,
where $\epsilon_{n}$ is given by \eqref{eq:epsn def}, we have $$\Vc(X)=o( r^4).$$

\item
If $f$ is $\epsilon$-ultraflat then for all $r\gg 1/n^{1/2-4\epsilon}$ we have $$\Vc(X) \ll r_2(n)^{-\epsilon} r^4.$$

\end{enumerate}

\end{corollary}

\subsection{Outline of the paper}

In section \ref{sec::unif distr BR} we give a proof to Theorem \ref{thm:unif distr BR}.
In section \ref{sec:Berry limitations}
we construct a counterpoint to Theorem \ref{thm:unif distr BR}: a sequence of eigenfunctions corresponding to a density one sequence
of energy levels, and balls with radii satisfying \eqref{eq:r0>>E^1/2} that do not possess a ``fair" share of the $L^{2}$-mass
(see Corollary \ref{cor:not equidist ce} and also Remark \ref{rem:non-equidist}). 
Section \ref{sec:var estimates} is dedicated to the proofs of Theorem \ref{thm:gen n var estimate}, and Corollaries \ref{cor: Locally flat} and \ref{cor: Flat result}. The proof of Theorem \ref{thm:close pairs} is given towards the end of section \ref{sec:close pairs proof}, 
after some considerable preparatory work.
The proofs of theorems \ref{thm:BR exceptions} and \ref{thm:G* asymp} 
will be given in sections \ref{sec:ref aut} and \ref{sec:NumAutoClasses} respectively; these are straightforward applications
of the more general theorems \ref{thm:not BR seq asympt Andrew} and \ref{NumAutoClasses} respectively.

\subsection{Acknowledgements} The authors would like to thank Mike Bennett, Valentin Blomer,
Stephen Lester, Ze\'{e}v Rudnick, Mikhail Sodin, and Peter Sarnak for a number of stimulating
and fruitful conversations and their remarks. The research leading to these results has received funding from the
European Research Council under the European Union's Seventh
Framework Programme (FP7/2007-2013), ERC grant agreement n$^{\text{o}}$ 670239  (A.G.)
and n$^{\text{o}}$ 335141 (I.W.), as well as from NSERC Canada under the CRC program (A.G.).

\section{Proof of Theorem \ref{thm:unif distr BR}}

\label{sec::unif distr BR}

The following lemma gives an exact formula for the error term and so will be useful in the   proof of Theorem \ref{thm:unif distr BR}, and beyond.

\begin{lemma}
\label{lem:L2 mass sum Bessel}
Let $f_{n}$ be given by \eqref{eq:f(x)=sum c*exp} with \eqref{eq:sum clamsqr=1} satisfied, $x\in \T$ and $r>0$. We have
the identity
\begin{equation*}
\int\limits_{B_{x}(r)}f_{n}^{2}dy - \pi r^{2}   = 2\pi r^{2}\sum\limits_{\lambda\ne \lambda'}c_{\lambda}\overline{c_{\lambda}'}
e(\langle x, \lambda-\lambda' \rangle)\frac{J_{1}(r\|\lambda-\lambda'\|)}{r\|\lambda-\lambda'\|},
\end{equation*}
where $J_{1}$ is the a Bessel function of the first kind.
\end{lemma}

We note that $J_1(t)$ oscillates between positive and negative values, and that for all $T>0$,
\begin{equation} \label{BesselDecay}
\max_{T\le t <2T } |J_{1}(t)|\asymp \min \left\{ T, \frac{1}{T^{1/2}} \right\}.
\end{equation}

Before giving a proof of Lemma \ref{lem:L2 mass sum Bessel} we formulate the following
corollary establishing an explicit relation between the closest pairs of lattice points and
radii satisfying the equidistribution \eqref{eq:unif equidist sup r>r0}, of independent interest, towards proving Theorem \ref{thm:unif distr BR}.

\begin{corollary}
\label{cor:UpperBound}
Given $f_{n}\in \Ec_{n}$ as in \eqref{eq:f(x)=sum c*exp}, satisfying \eqref{eq:sum clamsqr=1}, $x\in \T$, $P\ge 1$ sufficiently large,
and
\begin{equation}
\label{eq:r>P*r2^2/3/mindist}
r \geq H \cdot \frac {r_2(n)^{2/3}} { \min_{\lambda\ne \lambda'}  \|\lambda-\lambda'\| } ,
\end{equation}
we have
\begin{equation}
\label{eq:L2mass=pir^2+O}
\int\limits_{B_{x}(r)}f_{n}^{2}dy =\left\{ \pi + O\left(\frac{1}{H^{3/2}}\right) \right\}   r^{2},
\end{equation}
where the constant involved in the $`O'$-notation is absolute.
\end{corollary}

Corollary \ref{cor:UpperBound} yields that for $H=H(n)\rightarrow\infty$ we have
\begin{equation*}
\int\limits_{B_{x}(r)}f_{n}^{2}dy =\left\{ \pi + o(1) \right\} r^{2},
\end{equation*}
uniformly for all $r$ satisfying \eqref{eq:r>P*r2^2/3/mindist}.

\begin{proof}[Proof of Corollary \ref{cor:UpperBound} assuming Lemma \ref{lem:L2 mass sum Bessel}]
Let $$R= r \cdot \min_{\lambda\ne \lambda'}  \|\lambda-\lambda'\|$$ so that
\eqref{eq:r>P*r2^2/3/mindist} is
\begin{equation}
\label{eq:R>Pr2^2/3}
R \ge H\cdot r_2(n)^{2/3}.
\end{equation}
Lemma \ref{lem:L2 mass sum Bessel} together with \eqref{BesselDecay} then yield
\begin{equation}
\label{eq:mass-exp<<R^-3/2r^2 sum}
\left|\ \int\limits_{B_{x}(r)}f_{n}^{2}dy - {\pi r^{2}}  \right|
 \ll  r^2 \cdot R^{-3/2} \left( \sum\limits_{\lambda }
 |c_{\lambda} | \right)^2.
\end{equation}
But $$\left( \sum_{\lambda } |c_{\lambda} | \right)^2\leq  \sum_{\lambda} 1\cdot  \sum_{\lambda } |c_{\lambda} |^2 = r_2(n)$$
by \eqref{eq:sum clamsqr=1} and the Cauchy-Schwarz inequality, so the upper bound in \eqref{eq:mass-exp<<R^-3/2r^2 sum} is
\begin{equation*}
\left|\ \int\limits_{B_{x}(r)}f_{n}^{2}dy - {\pi r^{2}}  \right| \ll r^2 \cdot  \frac{r_2(n)}{R^{3/2}} \ll
\frac{r^2}{H^{3/2}}
\end{equation*}
by \eqref{eq:R>Pr2^2/3}. The latter inequality is precisely the statement \eqref{eq:L2mass=pir^2+O} of Corollary \ref{cor:UpperBound}.
\end{proof}

\begin{proof}[Proof of Theorem \ref{thm:unif distr BR}]

Since we assumed the $BR(\delta)$ condition \eqref{eq:BR(delta)}, an application of Corollary \ref{cor:UpperBound} with
$H=n^{\eta}$ yields that
\begin{equation}
\label{eq:L2mass-pir^2<<n-2/3eta}
\left| \frac{ \int\limits_{B_{x}(r)}f(y)^{2}dy}{\pi r^{2}} - 1 \right| \ll n^{-3\eta/2}
\end{equation}
holds uniformly for all
\begin{equation*}
r> n^{-1/2+(\delta+\eta)} \cdot r_{2}(n)^{2/3}.
\end{equation*}
That \eqref{eq:L2mass-pir^2<<n-2/3eta}, in particular, holds for all $r>n^{-1/2+\epsilon}$, as claimed \eqref{eq:L2 disc < n^-3/2eta}
in Theorem \ref{thm:unif distr BR}, follows from \eqref{eq:N=n^o(1)} and our assumption $0<\eta < \epsilon-\delta.$

\end{proof}

Now we finally prove Lemma \ref{lem:L2 mass sum Bessel}.

\begin{proof}[Proof of Lemma \ref{lem:L2 mass sum Bessel}]
Upon multiplying \eqref{eq:f(x)=sum c*exp} with its conjugate, and separating the diagonal summands from the off diagonal,
for $x\in\T$ and $r>0$ we have that
\begin{equation*}
\int\limits_{B_{x}(r)}f(y)^{2}dy = \pi r^{2} +
\sum\limits_{\lambda\ne \lambda'}c_{\lambda}\overline{c_{\lambda'}} \int\limits_{B_{x}(r)}e(\langle \lambda-\lambda',y \rangle)dy.
\end{equation*}
Therefore, transforming the variables $y=r\cdot z + x $ with $z\in B_{0}(1)$, we have
\begin{equation*}
\int\limits_{B_{x}(r)}f(y)^{2}dy - \pi r^{2}  = r^{2}\sum\limits_{\lambda\ne \lambda'} c_{\lambda}\overline{c_{\lambda'}}e(\langle \lambda-\lambda',x \rangle)
\int\limits_{B(1)}e(\langle r(\lambda-\lambda'),z \rangle)dz.
 \end{equation*}
where this time $B(1)\subseteq \R^{2}$ is the {\em Euclidian} centred unit ball.
This yields the identity
\begin{equation}
\label{eq:L2 mass sum lambda lambda'}
\int\limits_{B_{x}(r)}f(y)^{2}dy - \pi r^{2} = r^{2}\sum\limits_{\lambda\ne \lambda'}
c_{\lambda}\overline{c_{\lambda'}}e(\langle \lambda-\lambda',x \rangle) \cdot  \widehat{\chi}(r(\lambda-\lambda')),
\end{equation}
where $\chi$ is the characteristic of the unit disc. As $\chi$ is rotationally invariant so is its Fourier transform;
a direct computation shows that its Fourier transform is given explicitly by
\begin{equation*}
\widehat{\chi}(\xi) = 2\pi\frac{J_{1}(\|\xi\|)}{\|\xi\|}.
\end{equation*}
Substituting the latter into \eqref{eq:L2 mass sum lambda lambda'} yields the statement of Lemma \ref{lem:L2 mass sum Bessel}.
\end{proof}

Next we prove a strong version of Berry's conjecture for toral Laplace eigenfunctions.

\begin{corollary}
\label{cor:UpperBound2} For almost all $n\in S$, if  $f_{n}$ is as in \eqref{eq:f(x)=sum c*exp}, satisfying \eqref{eq:sum clamsqr=1},  and
\[
r \geq \frac {   (\log n)^{1+ \frac {\log 2}3 +\epsilon}  } {\sqrt{n }} ,
\]
then for all $x\in \T$ we have
\begin{equation*}
\int\limits_{B_{x}(r)}f_{n}^{2}dy =\{ \pi + o(1) \}   r^{2}.
\end{equation*}
\end{corollary}

\begin{proof}

If $n\in S$ then we can write $n=Nm^2$, in which $m$ has only prime factors $\equiv 3 \pmod 4$, and $N$ has no such prime factors, and then $r_2(n)=r_2(N)$. Note that if $N=2^k\ell$ where $\ell$ is odd, then $r_2(N)=4\tau(\ell)$, where $\tau(.)$ is the divisor function. If $N$ is squarefree then $\tau(N)=4\cdot 2^{\omega_o(N)}$ where $\omega_o(N)$ denotes the number of distinct {\em odd} prime factors of $N$. A famous result of Hardy and Ramanujan states that
$$\omega(N)=\left\{ 1+o(1)\right\} \log\log N$$
for almost all integers $N$. However our integers $N$ only have odd prime factors that are $\equiv 1 \pmod 4$
so $$\omega(N)=\left\{ \frac 12+o(1)\right\} \log\log N$$ for almost all such integers $N\in S$.
Since most integers have only a small part involving squares, one can then deduce that for almost all integers $n\in S$, one has
\begin{equation}
\label{r2aa}
r_2(n) = 2^{ \left\{ \frac 12+o(1)\right\} \log\log n} = (\log n)^{ \frac {\log 2}2+o(1)}
\end{equation}

As discussed after its statement, Theorem \ref{thm:G* asymp}
implies that for almost all $n\in S$, all $\lambda\ne \lambda'\in \Ec_{n}$ satisfy
\begin{equation}
\label{eq:lamb-lamb'>>sqrtn/logn}
\|\lambda-\lambda'\|\gg \frac{\sqrt{n}}{\log n};
\end{equation}
we conjecture that this is ``best possible'', though this is irrelevant here.
The statement of Corollary \ref{cor:UpperBound2} then follows upon substituting the above two results \eqref{r2aa} and \eqref{eq:lamb-lamb'>>sqrtn/logn} into Corollary \ref{cor:UpperBound}.
\end{proof}

\section{Limitations on Berry's conjecture}

\label{sec:Berry limitations}

Next we prove the counterpoint to Corollary  \ref{cor:UpperBound}:

\begin{proposition}
\label{prop:not equidist ce}
\label{cor:LowerBound} For all $n\in S$ there exists an $f_{n}$ as in \eqref{eq:f(x)=sum c*exp}, satisfying \eqref{eq:sum clamsqr=1},
and a value of
\[
r \gg \frac {  1  } { \min_{\lambda\ne \lambda'}  \|\lambda-\lambda'\| }
\]
for which
\begin{equation*}
\left|\int\limits_{B_{0}(r)}f_{n}^{2}dy -\pi     r^{2} \right| \gg r^2 .
\end{equation*}
In fact we get this lower bound for almost all $x\in \T$.
\end{proposition}

\begin{proof}  Select $\lambda \ne \lambda'\in \Ec(n)$ for which $|\lambda - \lambda'|$ is minimal.
Let $c_j=0$ unless $j=\lambda, \lambda',\overline{\lambda}$ or $\overline{\lambda'}$, in which case we have $c_j=1/2$ (with obvious modifications if $\lambda$ or $\lambda'\in \mathbb R$ or $\lambda'=\overline{\lambda}$) . By Lemma \ref{lem:L2 mass sum Bessel}, we then have
\begin{equation*}
\int\limits_{B_{x}(r)}f_{n}^{2}dy - \pi r^{2}   =  \pi r^{2}
\cos (2\pi \langle x, \lambda-\lambda' \rangle)\frac{J_{1}(r\|\lambda-\lambda'\|)}{r\|\lambda-\lambda'\|}.
\end{equation*}
By \eqref{BesselDecay} we deduce that there exists $r\asymp 1/\|\lambda-\lambda'\|$ for which
\begin{equation*}
\left| \int\limits_{B_{x}(r)}f_{n}^{2}dy - \pi r^{2}   \right| \asymp  \pi r^{2}
| \cos (2\pi \langle x, \lambda-\lambda' \rangle)| .
\end{equation*}
The right-hand side will be big for most choices of $x$, but, in particular, taking $x=0$ we obtain
$$\left| \int\limits_{B_{0}(r)}f_{n}^{2}dy - \pi r^{2}   \right| \asymp r^2.$$
\end{proof}

\begin{corollary}
\label{cor:not equidist ce}
For almost all $n\in S$, there exists an $f_{n}$ as in \eqref{eq:f(x)=sum c*exp}, satisfying \eqref{eq:sum clamsqr=1}, and a value of
\[
r \geq \frac {   (\log n)^{ \frac {\log 2}2 -\epsilon}  } {\sqrt{n }} ,
\]
for which
\begin{equation*}
\int\limits_{B_{0}(r)}f_{n}^{2}dy -\pi     r^{2} \gg r^2 .
\end{equation*}
In fact we get this lower bound for almost all $x\in \T$.
\end{corollary}

\begin{proof}  There are $r_2(n)$ elements of $\Ec(n)$ on a circle of perimeter $2\pi \sqrt{n}$, and so $$\min_{\lambda\ne \lambda'}  \|\lambda-\lambda'\|  <2\pi \cdot\frac{\sqrt{n}}{r_2(n)}.$$ We substitute this bound into Proposition \ref{prop:not equidist ce} to obtain the lower bound
$r\gg r_2(n)/\sqrt{n}$ for all $n\in S$. The result now follows from \eqref{r2aa}.
\end{proof}

\begin{remark}
\label{rem:non-equidist}
We can infer from Corollaries
\ref{cor:UpperBound2} and \ref{cor:not equidist ce} that our interpretation of Berry's conjecture is generically true for 
$$r>\frac{(\log n)^A}{\sqrt{n}},$$ for any
$$A>1+\frac{\log 2}3=1.23104906\ldots,$$ and generically false for $$r>(\log n)^B/\sqrt{n},$$ with
$$B< \frac{\log 2}2=0.34657359\ldots.$$ We would guess that there exists some critical exponent $C>0$ such that the conjecture is
generically true for $$r>\frac{(\log n)^A}{\sqrt{n}}$$ for every $A>C$, and is generically false for 
$$r\leq \frac{(\log n)^B}{\sqrt{n}}$$ for every $B<C$.
However we do not have a guess for the value of $C$.

\vspace{2mm}

It should be possible to improve Corollary \ref{cor:not equidist ce} with the exponent $$\frac{\log 3}2=0.54930614\ldots$$ in place of $\frac{\log 2}2$, as follows:\ Almost all $n\in S$ can be written as $Nm$ where $N$ is product of distinct primes $\equiv 1 \pmod 4$, and $N$ has a particular structure: It consist of $(\frac 12 -o(1)) \log\log n$ prime factors each of which lies in the interval
$$\left[\exp( (\log n)^{o(1)}), \ \exp( (\log n)^{1-o(1)})\right].$$ We split this interval into dyadic intervals, and run though the integers $N$ composed of such primes.  If $p=a^2+b^2$ then the $a+ib$ should be more-or-less equidistributed in angle, so the set of elements of $\Ec(N)$, in which $N$ has exactly $k$ prime factors can be modelled by the random model
\[
\left\{   \sqrt{N} \ e \left( \sum_{i=1}^k \delta_i \phi_i \right);\ \delta_1,\ldots,\delta_k\in \{ -1,1\} \ \right\}
\]
where each $\phi_i$ is an iirv, uniformly distributed in $\mathbb R/\mathbb Z$.

Suppose that $\lambda,\lambda'$ are the closest two elements of
$\Ec_{N}$. If $|\alpha|^2=m$ then $\alpha\lambda,\alpha\lambda'\in \Ec(n)$, and
$$|\alpha\lambda-\alpha\lambda'|/\sqrt{n}=| \lambda- \lambda'|/\sqrt{N}.$$ Now this, according to the random model, should be roughly the expected value of the minimum of
\[
\left\{      \left| \sum_{i=1}^k \eta_i \phi_i \right|: \ \eta_1,\ldots,\eta_k\in \{ -1,0,1\} \ \right\}
\]
(where $2\eta_i=\delta_i-\delta_i'$). We can use Fourier analysis to ask for the expected number of such elements in an interval $[-\epsilon, \epsilon]$. As these are iirv's, all but the main term disappears, and so the answer as $2\epsilon\cdot 3^k$. Therefore we should be able to take $$\epsilon\approx 3^{-k}=(\log n)^{-\frac{\log 3} 2 +o(1)},$$ and so the claim. One would expect this to be unconditionally provable using the second moment method, though we leave this for other authors.
\end{remark}

\section{Proof of Theorem \ref{thm:gen n var estimate} and Corollaries \ref{cor: Locally flat} and \ref{cor: Flat result}}
\label{sec:var estimates}

\begin{lemma}
\label{lem:var sum lambda} Assume  \eqref{eq:sum clamsqr=1}. The variance \eqref{eq:Var(X) def} of $X$ is given by
 \begin{equation*}
\Vc(X) = 8\pi^{2}r^{4}\sum\limits_{\lambda\ne \lambda'} |c_{\lambda}c_{\lambda'}|^{2}
\frac{J_{1}(r\|\lambda-\lambda'\|)^{2}}{r^{2}\|\lambda-\lambda'\|^{2}}.
\end{equation*}
\end{lemma}

\begin{proof}[Proof of Theorem \ref{thm:gen n var estimate} assuming Lemma \ref{lem:var sum lambda}]

We invoke Lemma \ref{lem:var sum lambda} and separate the near-diagonal terms
$$0< |\lambda-\lambda'| < \frac{1}{\xi r}$$ from the rest to yield
\begin{equation*}
\begin{split}
\frac{\Vc(X)}{r^{4}} &\ll \sum\limits_{0< |\lambda-\lambda'| < \frac{1}{\xi r}} |c_{\lambda}c_{\lambda'}|^{2}
\frac{J_{1}(r^{2}\|\lambda-\lambda'\|^{2})}{r^{2}\|\lambda-\lambda'\|^{2}} \\&+
\sum\limits_{|\lambda-\lambda'| \ge \frac{1}{\xi r}} |c_{\lambda}c_{\lambda'}|^{2}
\frac{J_{1}(r^{2}\|\lambda-\lambda'\|^{2})}{r^{2}\|\lambda-\lambda'\|^{2}}.
\end{split}
\end{equation*}
Upon using the bound $J_{1}(t) \ll t$ for the range $|\lambda-\lambda'| < \frac{1}{\xi r},$ and
the bound $J_{1}(t) \ll \frac{1}{\sqrt{t}}$ for $|t|\geq 1/\xi$ (see \eqref{BesselDecay}), we obtain the estimate
\begin{equation}
\label{eq:V(X)/r4<<sum nondiag}
\frac{\Vc(X)}{r^{4}} \ll \sum\limits_{0< |\lambda-\lambda'| < \frac{1}{\xi r}} |c_{\lambda}c_{\lambda'}|^{2}+
\sum\limits_{|\lambda-\lambda'| \ge \frac{1}{\xi r}} |c_{\lambda}c_{\lambda'}|^{2}
\frac{1}{r^{3}\|\lambda-\lambda'\|^{3}}.
\end{equation}
For the latter summation in \eqref{eq:V(X)/r4<<sum nondiag} we have
\begin{equation*}
\sum\limits_{|\lambda-\lambda'| \ge \frac{1}{\xi r}} |c_{\lambda}c_{\lambda'}|^{2}
\frac{1}{r^{3}\|\lambda-\lambda'\|^{3}} \le \xi^{3} \cdot \sum\limits_{\lambda,\lambda'\in\Ec_{n}} |c_{\lambda}c_{\lambda'}|^{2}
= \xi^{3}
\end{equation*}
by \eqref{eq:sum clamsqr=1}. The result follows.
\end{proof}

\begin{proof}[Proof of Lemma \ref{lem:var sum lambda}]
By Lemma \ref{lem:L2 mass sum Bessel} we have
\begin{equation}
\label{eq:X-EX^2 sum lambda1234}
\begin{split}
(X_x-\E[X_x])^{2} = 4\pi^{2} r^{4 }
\sum\limits_{\lambda\ne \lambda',\lambda''\ne \lambda'''}&c_{\lambda}\overline{c_{\lambda'}}c_{\lambda''}
\overline{c_{\lambda}'''}
e(\langle x, \lambda-\lambda'+\lambda''-\lambda''' \rangle)\times\\&\times\frac{J_{1}(r\|\lambda-\lambda'\|)J_{1}(r\|\lambda''-\lambda'''\|)}
{r^{2}\|\lambda-\lambda'\|\cdot \|\lambda''-\lambda'''\|}.
\end{split}
\end{equation}
Integrating w.r.t. $x\in\T$ we are only left with the diagonal:
\begin{equation*}
\begin{split}
\Vc(X) &=   \int\limits_{\T}(X_x-\E[X_x])^{2}dx \\ &= 4\pi^{2} r^{4 }
\sum\limits_{\substack{\lambda\ne \lambda',\lambda''\ne \lambda''' \\ \lambda-\lambda'+\lambda''-\lambda'''=0}} c_{\lambda}\overline{c_{\lambda'}}c_{\lambda''}
\overline{c_{\lambda}'''}  \frac{J_{1}(r\|\lambda-\lambda'\|)J_{1}(r\|\lambda''-\lambda'''\|)}
{r^{2}\|\lambda-\lambda'\|\cdot \|\lambda''-\lambda'''\|}.
\end{split}
\end{equation*}
Now, as $\lambda\ne \lambda'$ and $\lambda''\ne \lambda'''$ we have $\lambda-\lambda'+\lambda''-\lambda'''=0$ if and only if either ($\lambda=-\lambda''$ and $\lambda'=-\lambda'''$) or
($\lambda=\lambda'''$ and $\lambda'=\lambda''$). Using this together with \eqref{eq:clambda,c-lambda} we deduce 
the statement of Lemma \ref{lem:var sum lambda}.
\end{proof}

\subsection{Proof of Corollaries \ref{cor: Locally flat} and \ref{cor: Flat result}}

\begin{proof} [Proof of Corollary \ref{cor: Locally flat}]

An application of Theorem \ref{thm:gen n var estimate} with $\xi=\epsilon$ yields the bound
\begin{equation}
\label{eq:V(X)<<sum l-l'<eta+eps^3}
\Vc(X) \ll \left(\sum\limits_{0<|\lambda-\lambda'| < \eta\sqrt{n}} |c_{\lambda}|^{2}|c_{\lambda'}|^{2} +\epsilon^{3}\right)r^{4} .
\end{equation}
Now
\begin{equation*}
\begin{split}
\sum\limits_{0<|\lambda-\lambda'| < \eta\sqrt{n}} |c_{\lambda}|^{2}|c_{\lambda'}|^{2} &\le
\sum\limits_{\lambda\in\Ec_{n}}|c_{\lambda}|^{2}\sum\limits_{|\lambda-\lambda'| < \eta\sqrt{n}}|c_{\lambda'}|^{2} \le
\epsilon \sum\limits_{\lambda\in\Ec_{n}}|c_{\lambda}|^{2} = \epsilon,
\end{split}
\end{equation*}
by \eqref{eq:Being Flat} and \eqref{eq:sum clamsqr=1}.
The statement of Corollary \ref{cor: Locally flat} then follows upon substitute the latter inequality into \eqref{eq:V(X)<<sum l-l'<eta+eps^3}.

\end{proof}

\begin{proof} [Proof of Corollary \ref{cor: Flat result}]

The assumption that $f$ is flat implies that \eqref{eq:Being Flat} holds with $\eta= n^{ -\epsilon_n}$.
A straightforward application of Corollary \ref{cor: Locally flat} with $\eta=n^{ -\epsilon_n}$ and any fixed $\epsilon$ yields
the first statement of this corollary.

For the second part we apply Theorem \ref{thm:gen n var estimate} with $\xi= r_{2}(n)^{  -\epsilon}$ to yield the bound
\begin{equation}
\label{eq:Var<<sum close +r-eps}
\Vc(X) \ll \left(\sum\limits_{0 < |\lambda-\lambda'| < r_{2}(n)^{\epsilon}/r}|c_{\lambda}c_{\lambda'}|^{2}   +
r_{2}(n)^{-3\epsilon}\right) \cdot r^{4},
\end{equation}
and for $\epsilon$-ultraflat functions \eqref{eq:clam ultraflat} we have
\begin{equation}
\label{eq:sum<<close pairs ultraflat}
\begin{split}
\sum\limits_{0 < |\lambda-\lambda'| < r_{2}(n)^{\epsilon}/r}|c_{\lambda}c_{\lambda'}|^{2} \le
\frac{1}{r_{2}(n)^{2-2\epsilon}}\# \left\{ \lambda\ne\lambda'\in\Ec_{n}:\: |\lambda-\lambda'| < \frac{r_{2}(n)^{\epsilon}}{r}\right\}.
\end{split}
\end{equation}
For $$r\gg \frac{1}{n^{1/2-4\epsilon}}\ge \frac{r_{2}(n)^{\epsilon}}{n^{1/2-3\epsilon}}$$ we may bound the latter as
\begin{equation*}
\begin{split}
&\# \left\{ \lambda\ne\lambda'\in\Ec_{n}:\: |\lambda-\lambda'| < \frac{r_{2}(n)^{\epsilon}}{r}\right\}
\\&\le \#\left\{ \lambda\ne\lambda'\in\Ec_{n}:\: |\lambda-\lambda'| < n^{1/2-3\epsilon}\right\} \ll r_{2}(n)^{2-3\epsilon}
\end{split}
\end{equation*}
by Theorem \ref{thm:close pairs}. The result finally follows upon substituting the latter estimate into \eqref{eq:sum<<close pairs ultraflat},
and then finally into \eqref{eq:Var<<sum close +r-eps}.

\end{proof}

\section{Close lattice points on a  given circle}
\label{sec:close pairs proof}

Our goal is to prove Theorem \ref{thm:close pairs}.
Our proof yields the more explicit upper bound,
\begin{equation*}
\#\left\{ \alpha,\beta \in \Ec_{n}:\: |\alpha-\beta| \le n^{1/2-\epsilon}\right\} \ll
r_{2}(n)^{2-\epsilon}  + \frac{1}{\epsilon}\cdot  r_{2}(n)^{2-2\epsilon}.
\end{equation*}
In particular if $\epsilon=\frac{ \log\log |\Ec_{n}|}{\log |\Ec_{n}|}$ the bound is
\begin{equation*}
\#\left\{ \alpha,\beta \in \Ec_{n}:\: |\alpha-\beta| \le n^{1/2-\epsilon}\right\} \ll
\frac{r_{2}(n)^{2}}{\log r_{2}(n)}.
\end{equation*}

We will also show, using the result of Cilleruelo and Cordoba \cite{CC}, that we can replace the ``$-\epsilon$"
in the exponent on the right-hand side of \eqref{eq:close<n^1/2-eps<<r2^2-eps} by ``$-\tau\epsilon $" for any fixed $\tau<4$.

\subsection{The structure of the sets $\Ec_{n}$.} If $\alpha\in \Ec_{n}$ then so are $u\alpha$ for each  $u\in \CU:=\{ 1,-1, i, -i\}$, the set of units of $\mathbb Z[i]$. Note that there is therefore a unique $u\alpha=a+ib$ in first quadrant, so that $a>0$ and $b\geq 0$. We now describe the structure of  the quotient set
\[
  \Ec^*(n) : = \Ec_{n}/\CU.
\]
The key observation is that these sets are multiplicative; that is, $$\Ec^*(mn) = \Ec^*(m) \cdot \Ec^*(n)$$ if $(m,n)=1$ and all of the products are distinct, and so, in particular, $$r_2(mn)/4=(r_2(m)/4) \cdot (r_2(n)/4).$$ Therefore to fully understand the sets $\Ec_{n}$ we need only focus on $\Ec^*(p^k)$.
If $p\equiv 3 \pmod 4$ then $\Ec^*(p^k)=\emptyset$ if $k$ is odd, and $\Ec^*(p^k)=\{ p^{k/2}\}$ if $k$ is even.
Also $$\Ec^*(2^k)=\{ (1+i)^{k}\}$$ for all $k\geq 1$. If $p\equiv 1 \pmod 4$ then, as is well known, there are integers $a,b$, unique  up to sign and swapping their order, for which $p=a^2+b^2$. Therefore if $P=a+ib$ then $\Ec^*(p)=\{ P,\overline{P}\}$, and
$$\Ec^*(p^k)=\{ P^k,\ P^{k-1}\overline{P},\ldots, P \overline{P}^{k-1},  \overline{P}^k  \}.$$

\subsection{A first bound, using the structure}
Throughout this section we will assume, without loss of generality,
that
\begin{equation}
\label{eq:n=prod prim pow 1}
n=\prod_{i=1}^k p_i^{n_i}
\end{equation}
where each $p_j\equiv 1 \pmod 4$, at first in no particular order, then
later for non-increasing $\left\{\frac{\log(n_{i}+1)}{n_{i}\log{p_{i}}}\right\}$
(see section \ref{sec:balance prime powers}).

\begin{lemma}
\label{L2.1}
For every $n$ of the form \eqref{eq:n=prod prim pow 1}, let $m$ be given by
\begin{equation}
\label{eq:m=prod coprim n/m}
m=\prod_{i=1}^\ell p_i^{n_i}
\end{equation}
and $\theta\in \mathbb R/\mathbb Z$. The number of
$\lambda\in \Ec_{n}$ satisfying
\begin{equation}
\label{eq:lambda-exp(th)<sqrt(n/2m)}
|\lambda-\sqrt{n}e^{2i\pi \theta}|<\sqrt{ n/2m}
\end{equation}
is $\leq 4\prod_{i=\ell+1}^k (n_i+1).$
 \end{lemma}

\begin{proof} Suppose there are $> 4\prod_{i=\ell+1}^k (n_i+1)$ numbers $\lambda\in \Ec_{n}$ satisfying
\eqref{eq:lambda-exp(th)<sqrt(n/2m)} for some $\theta\in \mathbb R/\mathbb Z$.
We write each $\lambda\in \Ec_{n}$ as
\[
\lambda:= u \prod_{i=1}^k P_i^{e_i} \overline{P_i}^{n_i-e_i}.
\]
For at least two of the $\lambda$ with $$|\lambda-\sqrt{n}e^{2i\pi \theta}|<\sqrt{ n/2m},$$ the
$u$, and the $e_i$ are the same for all $i>\ell$, by the pigeonhole principle; we write the two numbers
as $\lambda=\gamma\beta$ and $\lambda'=\gamma\beta'$ where
$$\gamma:= u \prod_{i=\ell+1}^k P_i^{e_i} \overline{P_i}^{n_i-e_i}.$$ Now $\beta-\beta'\in \Z[i]$ with
$|\beta|=|\beta'|$ so that $|\beta-\beta'|\geq \sqrt{2}$. Therefore
\begin{equation*}
\begin{split}
2\sqrt{ n/2m}&>|\lambda-\sqrt{n}e^{2i\pi \theta}|+|\lambda'-\sqrt{n}e^{2i\pi \theta}|\geq |\lambda-\lambda'|\\
& \geq \sqrt{2} |\gamma| =\sqrt{2}  \prod_{i=\ell+1}^k p_i^{n_i/2} =\sqrt{2n/m} ,
\end{split}
\end{equation*}
a contradiction.
\end{proof}

We can revisit Lemma \ref{L2.1} putting a better lower bound on $|\beta-\beta'|$ by using Cilleruelo-Cordoba ~\cite{CC}:\

\begin{lemma} \label{L2.1b} Fix $\epsilon>0$, and let $n$ be of the form
\eqref{eq:n=prod prim pow 1} and $m$ of the form \eqref{eq:m=prod coprim n/m}.
Then for every $\theta\in \mathbb R/\mathbb Z$ the number of
$\lambda\in \Ec_{n}$ satisfying $$|\lambda-\sqrt{n}e^{2i\pi \theta}|<\sqrt{n}/m^{1/4-\epsilon}$$
is $\ll (1/\epsilon) \tau(n/m)$.
\end{lemma}

\begin{proof}
Cilleruelo and Cordoba \cite{CC} proved that an arc on a circle of radius $R$, which contains more than $2r$ lattice points, has length $> 2^{1/2} R^{r/(2r+1)}$. Therefore if our arc in the proof of Lemma \ref{L2.1}, contains  $>8r\prod_{i=\ell+1}^k (n_i+1)$  numbers $\lambda\in \Ec_{n}$, then we have more than $2r$ with the same $\gamma$, and so more than $2r$ lattice points $\beta$ lie  on a circle of radius $m^{1/2}$.
Thus the $\beta$-arc has width $> 2^{1/2} m^{r/2(2r+1)}$. We therefore deduce that there are $\leq 8r\prod_{i=\ell+1}^k (n_i+1)$  numbers $\lambda\in \Ec_{n}$ with $$|\lambda-\sqrt{n}e^{2i\pi \theta}|<\frac{(n/2)^{1/2}}{m^{(r+1)/2(2r+1)}},$$ which is $\leq  n^{1/2}/m^{1/4-\epsilon} $ if $r\gg 1/\epsilon$.
\end{proof}

\subsection{Balancing a prime power and its power}
\label{sec:balance prime powers}

For the rest of this section we will organize the $p_i$ so that the $\frac{\log(n_i+1)}{n_i\log p_i}$ are non-increasing. This will allow us to generalize the above argument to the case in which $m$ and $n/m$ are not necessarily coprime.

\begin{corollary} \label{C2.2}
If $m\geq n^{ \epsilon}$ or $n_{\ell+1}\ll 1/\epsilon$ then
\[
\#\{ (u,v)\in \Ec_{n}:\ |u-v|\leq n^{1/2-\epsilon}\}   \ll_\epsilon |\Ec_{n}|^{2-\epsilon}+
\frac{1}{\epsilon}\cdot |\Ec_{n}|^{2-2\epsilon}.
\]
 \end{corollary}

\begin{proof}
Putting $$m^*=mp_{\ell+1}^{n_{\ell+1}}$$ we have $m^{*}> n^{2\epsilon}/2$ by the definition of $m$.
If $n/m=n^\delta$ then $\tau(n/m)\leq \tau(n)^\delta$.
To see this define $e_i$ to satisfy the equation $n_i+1=(p_i^{n_i})^{e_i}$, note that  the $e_i$ are ordered so that  $e_1\geq e_2\geq \ldots$. Then
$\tau(n/m)\ = \prod_{i=\ell+1}^k (p_i^{n_i})^{e_i} = (n/m)^{E}$ say, and $\tau(m)=m^F$, where $E\leq e_{\ell+1}\leq e_{\ell } \leq F$,   and the claim follows.

Now  $n/m^*\leq 2n^{1-2\epsilon}$ and so $\tau(n/m^*)\leq 2\tau(n)^{1-2\epsilon}$,
by the argument in the first paragraph. This implies
$$\tau(n/m)=\tau(n/m^*)(n_{\ell+1}+1) \ll \tau(n)^{1-2\epsilon}/\epsilon.$$

If $m\geq n^{ \epsilon}$ then $n/m  \leq n^{1-\epsilon}$, and so $\tau(n/m)\leq \tau(n)^{1-\epsilon}$, by the first paragraph.
Therefore, by Lemma \ref{L2.1},
\[
\begin{split}
&\#\{ (u,v)\in \Ec_{n}:\   |u-v|\leq n^{1/2-\epsilon}\}   =  \sum_{u\in \Ec_{n}} \#\{ v\in \Ec_{n}:\  |u-v|\leq n^{1/2-\epsilon}\}  \\
& \leq r_2(n) \max_{\theta\in \mathbb R/\mathbb Z}
\#\{ v\in \Ec_{n}:\  |v-\sqrt{n}e^{2i\pi \theta}|\leq \sqrt{ n/2m}\}   \leq  r_2(n)   \cdot 4\tau(n/m),
\end{split}
\]
and the result follows from the bounds on $\tau(n/m)$ given above.
\end{proof}

\begin{proof} [Proof of Theorem \ref{thm:close pairs}] Corollary \ref{C2.2} yields the result at once,
unless $m< n^\epsilon $ and $n_{\ell+1}\geq 10/\epsilon$. In this case $p_{\ell+1}^{n_{\ell+1}}> n^\epsilon/2$, else $ n^{2\epsilon}/2<m^*=mp_{\ell+1}^{n_{\ell+1}} \leq m n^\epsilon/2$, so that
$m\geq n^{ \epsilon}$. For ease of notation, we write $q=p_{\ell+1}$     with $Q=P_{\ell+1}$, and $N=n_{\ell+1}$, so that $m^*=mq^N$:

We let $d$ be the largest integer for which $$m^\dag:=mq^d\leq n^{2\epsilon}/2.$$
We now show that $\epsilon N/2<d< N$:
By the definition of $N$, we know that $m^*>n^{2\epsilon}/2$ and so $d<N$.
 Now $N\geq 10/\epsilon$, and $q^N\leq n$, so $q<n^{\epsilon/10}$.  By definition $ q^{d+1}>n^{2\epsilon}/2m>n^{ \epsilon}/2>q^{10}/2$, and so $d\geq 9$. Moreover $$(q^{d+1})^{1/\epsilon+1}>(n^{ \epsilon}/2)^{1/\epsilon+1} \geq n\geq q^N,$$ and so $d\geq \epsilon N/2$

For a given integer $d$, let $w$ be the smallest integer with $wd\geq N$.  If we have $w+1$ integers amongst $0,\ldots,N$ then two of them differ by at most $d$. We have $w\leq 2/\epsilon+1$ since $d>\epsilon N/2$.

We now prove that there are $\leq 4(w+1)\tau(n/m^*)$ numbers $\alpha\in \Ec_{n}$ with $$|\alpha-\sqrt{n}e^{2i\pi \theta}|<\sqrt{ n/2m^\dag}$$ for every $\theta\in \mathbb R/\mathbb Z$: \ For if not then we have  $\alpha,\alpha'$ with $e_i=e_i'$ for all $i>\ell+1$, and $u=u'$, but the exponents of $Q$ and $\overline{Q}$ are  $Q^e\overline{Q}^{n-e}$ and
$Q^{e+\Delta}\overline{Q}^{n-e-\Delta}$, for some $\Delta, 0\leq \Delta\leq d$. The contribution to $\gamma$ is therefore
$Q^e\overline{Q}^{n-e-\Delta}$ which has norm $q^{\frac{n-\Delta}2}\geq q^{\frac{n-d}2}$.
Therefore $|\gamma|\geq \sqrt{n/m^\dag}$. We recover the same contradiction as in Lemma  \ref{L2.1}.

Proceeding as in the proof of  Corollary \ref{C2.2}, and recalling from these that $\tau(n/m^*)\leq 2\tau(n)^{1-2\epsilon}\leq 2r_2(n)^{1-2\epsilon}$, we then deduce that
\[
\#\{ (u,v)\in \Ec_{n}:\ |u-v|\leq n^{1/2-\epsilon}\}  \leq     8(w+1)\tau(n)^{2-2\epsilon}  \ll \frac{1}{\epsilon}\cdot |\Ec_{n}|^{2-2\epsilon}.
\]
We have proved both Theorem \ref{thm:close pairs} and the claim \eqref{eq:close pairs eps=loglogr2/logr2}.
\end{proof}

\begin{remark} We can improve Theorem \ref{thm:close pairs} unconditionally to
\begin{equation} \label{Bd}
\#\{ (u,v)\in \Ec_{n}:\ |u-v|\leq n^{1/2-\epsilon}\}   \ll_\epsilon |\Ec_{n}|^{2-\tau \epsilon} ,
\end{equation}
for any fixed $\tau<4$, by choosing $\ell$ so that $m\leq n^{\tau \epsilon}$, and using Lemma \ref{L2.1b}
in place of  Lemma  \ref{L2.1} in the proof above.
\end{remark}

\subsection{On  conjectured bounds for lattice points in short arcs} In Conjecture 15 of \cite{CA} it is conjectured that for any fixed $\epsilon>0$, there are $\ll_\epsilon 1$ lattice points on an arc of length $  R^{1-\epsilon}$ of a circle of radius $R$, in which case the upper bound in Theorem \ref{thm:close pairs} would be $\ll  |\Ec_{n}|$, for any fixed $\epsilon>0$.

In the special case that $n=p^{g}$ is a prime power, we can use ideas of Diophantine approximation to lower bound  $ |\textrm{Im}((a+ib)^{g})| $ where $a^2+b^2=p$:\
Let $$f(t):=\frac 1{2i}((t+i)^{g}-(t-i)^{g})$$ so that $$\textrm{Im}((a+ib)^{g})= F(a,b)$$ where
$F(x,y):=y^{g}f(x/y)$ is a homogenous polynomial of degree $g$. Now $$2if'(t)=g((t+i)^{ g-1}-(t-i)^{g-1})$$ and so $(f,f')=1$. This implies that $f$ has no repeated roots, and that $F$ has no repeated factors.

Roth's Theorem gives that $|f(a/b)|\gg_{g,\eta} 1/|b|^{2+2\eta}$, for each fixed $g$,  which implies that
\begin{equation}\label{eq: LB4P}
\textrm{Im}(P^{g})\gg_{g,\epsilon} p^{g/2} /|b|^{2+2\epsilon}\geq p^{g/2-1-\epsilon}.
\end{equation}
We can obtain a uniform version of this result by using the
$abc$-conjecture in the field $\mathbb Z[i]$ (see \cite{GS}): Suppose that $a+b=c$ with $a,b,c$ coprime elements of $\mathbb Z[i]$. Then
\[
    \prod_{Q|abc} |Q|   \gg_\epsilon \max \{ |a|, |b|, |c|\}^{1-\epsilon} ,
\]
where the product runs over the distinct primes $Q$ in $\mathbb Z[i]$. We have the equation
\[ (a+bi)^{g} - (a-bi)^{g}= 2i \ \textrm{Im}(P^{g}) , \]
and the terms are coprime if $p>2$. The $abc$-conjecture then implies
\[
\begin{split}
\textrm{Im}(P^{g})  p&= \textrm{Im}(P^{g}) (a+bi) (a-bi) \geq   \prod_{Q|(a+bi)^{g}  (a-bi)^{g}\cdot \ \textrm{Im}(P^{g})} |Q| \\
& \gg_\epsilon ( |a+ib|^{g}) ^{1-\epsilon} = ( p^{g/2}) ^{1-\epsilon},
\end{split}
\]
and therefore we recover \eqref{eq: LB4P}, in which the implicit constant is independent of $g$.

\section{Pairs of close-by lattice points, over all radii $ \leq \sqrt{N}$}

\subsection{Reformulation, and automorphisms of pairs of close lattice points}
\label{sec:ref aut}

If $a, b \in \Ec_{n}$ and $|a-b| \leq M$ then let $\alpha=\text{gcd}(a,b) $ 
(determined up to a unit, cf. \eqref{eq:sum close pairs alph,beta} below)
in $\mathbb Z[i]$, and $\beta=a/\alpha$. Then
$b=u\alpha\overline{\beta}$ where $(\beta,\overline{\beta})=1$ so that $\beta$ is not divisible by any integer $>1$ and $u\in \CU$. Therefore the elements
of $S(n,M)$, defined as the set
\begin{equation*}
S(n,M):= \{  (\alpha,\beta,u)\in \mathbb Z[i]^2\times \CU:\  |\alpha| \cdot |\beta| =\sqrt{n}, |\alpha| \cdot  |\beta-u\overline{\beta}|\leq M,\ (\beta,\overline{\beta})=1            \},
\end{equation*}
are in 1-to-1 correspondence with the pairs
$$\{ (a,b)\in \Ec_{n}: |a-b|\leq M\}.$$

\vspace{3mm}

Our goal is to estimate the number of exceptions to $BR(\delta)$, for given $\delta>0$. More generally,
given $N$ and $M \le 2\sqrt{N}$ define the {\em set of close pairs}
\begin{equation}
\label{eq:Gc def}
\Gc(N;M) = \{(\lambda,\lambda'):\:  \|\lambda\|^{2}=\|\lambda'\|^{2} \le N, \, 0< \|\lambda-\lambda'\| < M\}.
\end{equation}
Define the function $I(c):[0,1]\rightarrow\R$ by
\begin{equation}
\label{eq:I(c) def}
I(c):= \frac{4}{\pi}\int\limits_{0}^{1}(1-ct^{2})^{1/2}dt.
\end{equation}
Note that $I(c)$ is decreasing from $I(0)=\frac{4}{\pi}$ to $I(1)=1$, as $c$ goes from $0$ to $1$.
Moreover
\begin{equation}
\label{eq:I(c)=I(0)+O(c)}
I(c)=\frac{4}{\pi}+O(c).
\end{equation}

\begin{theorem}
\label{thm:not BR seq asympt Andrew}
Let $N\rightarrow\infty$ be a large parameter, $M=M(N)$, and $\Gc(N;M)$ defined in \eqref{eq:Gc def}.
We have
\begin{equation}
\label{eq:not BR seq asympt Andrew}
\#\Gc(N;M) = 4\cdot I\left(\frac{M^{2}}{4N}\right)\cdot \sqrt{N}M\log{M}\cdot \left(1+O\left( \frac{1}{\log{M}}\right)\right),
\end{equation}
which is an asymptotic as long as $M\rightarrow\infty$.
\end{theorem}

The proof of Theorem \ref{thm:not BR seq asympt Andrew} will be given in section \ref{sec:not BR seq asympt Andrew}.

\begin{proof} [Proof of Theorem \ref{thm:BR exceptions} assuming Theorem \ref{thm:not BR seq asympt Andrew}]
Let $J=[\sqrt{\log N}]$ and select $\epsilon>0$ so that $$(1-\epsilon)^J=1/2.$$
We write $N'=(1-\epsilon)N$ with $L= CN^{1/2-\delta}$ and $L'= C(N')^{1/2-\delta}$. Trivially we have
\[
 \Gc(N;L')-\Gc(N';L') \leq B^*(N)-B^*(N') \leq \Gc(N;L)-\Gc(N';L).
\]
Substituting in the estimate from Theorem \ref{thm:not BR seq asympt Andrew} we obtain
\[
B^*(N)-B^*((1-\epsilon)N) = \frac{4C}{\pi} (1-2\delta)  N^{1-\delta}\log{N} \cdot \epsilon\left(1+O\left( \epsilon \right)\right) .
\]
Replacing $N$ by $(1-\epsilon)^jN$ for $j=0,1,2,\ldots,J-1$ and summing, we obtain
\[
B^*(N)-B^*(N/2) = \frac{4C}{\pi} (1-2\delta)  \cdot \frac{ N^{1-\delta}-(N/2)^{1-\delta}} { (1-\delta)\epsilon  (1+O(\epsilon  ))}  \log{N} \cdot \epsilon\left(1+O\left( \epsilon \right)\right) .
\]
Finally replacing $N$ by $ N/2^j$ for $j=0,1,2,\ldots $ and summing, we obtain the claimed result.
\end{proof}

\subsection{The number of close-by pairs}
\label{sec:not BR seq asympt Andrew}

\begin{proof} [Proof of Theorem \ref{thm:not BR seq asympt Andrew}] We will use the ``if and only if'' criterion above, so we wish to count
\begin{equation}
\label{eq:sum close pairs alph,beta}
\begin{split}
\Gc(N,M) &= \sum\limits_{n\le N}|S(n,M)| \\=  \frac{1}{4}\sum_{u\in \{ 1,i,-1,-i\} } \#\bigg\{ \alpha, \beta \in \mathbb Z[i]:\  &|\beta| \leq \frac{\sqrt{N}}{|\alpha|}, |\beta-u\overline{\beta}|\leq \frac{M}{|\alpha|}, \ \  (\beta,\overline{\beta})=1 \bigg\},
\end{split}
\end{equation}
since $\alpha$ is determined up to a unit.
If $u=1$ then $|\beta-u\overline{\beta}|=2 | \text{Im}(\beta)|$, so if we write $\beta=x+iy$ then the conditions are
\begin{equation}
\label{eq:cond x,y u=1}
 |y|\leq M/2|\alpha|\ \ \text{and} \ \ x^2\leq N/ |\alpha|^2-y^2, \ \text{with} \ \ (x,y)=1  \ \text{and} \ x+y \ \text{odd}.
\end{equation}
To count this we fix $y$ and vary over $x$. Now $x$ runs through an interval of length $X$, say.
Moreover $$x\equiv y+1 \pmod 2,$$ and $(x,y)=1$.
Hence, by the inclusion exclusion principle, given $y$, the number of $x$ satisfying \eqref{eq:cond x,y u=1}
is
$$\frac{\varphi(2y)}{2y}\cdot X + O(\tau(y)),$$ where $\tau(y)$ denotes
the number of squarefree divisors of $y$. Therefore, in total, the summand on the r.h.s. of \eqref{eq:sum close pairs alph,beta}
corresponding to $u=1$ equals
\begin{equation}
\label{eq:summand u=1 sum}
\begin{split}
&\#\bigg\{ \alpha, \beta \in \mathbb Z[i]:\  |\beta| \leq \frac{\sqrt{N}}{|\alpha|}, |\beta-\overline{\beta}|\leq \frac{M}{|\alpha|}, \ \  (\beta,\overline{\beta})=1 \bigg\} \\&= 4 \sum\limits_{|\alpha|\le M/2}\sum_{1\leq y \leq M/2|\alpha|} \frac{\varphi(2y)}{2y}    \cdot \left(\frac N{ |\alpha|^2}-y^2 \right)^{1/2} + O\left(  \sum\limits_{|\alpha|\le M/2}  \sum\limits_{|y|\leq M/2|\alpha|}   \tau(y) \right).
\end{split}
\end{equation}
The inner summation of the error terms on the r.h.s. of \eqref{eq:summand u=1 sum} is
\begin{equation*}
\sum_{y:\  |y|\leq M/2|\alpha|}   \tau(y)\ll \frac{M}{2|\alpha|} \cdot \log(M/2|\alpha|).
\end{equation*}
Now, the number of such $\alpha$ with
$\frac{M}{2^{k+1}}<|\alpha|\leq \frac{M}{2^k}$ is $\ll \frac{M^2}{2^{2k}}$, so our bound for the total error term in \eqref{eq:summand u=1 sum} is
\begin{equation*}
\sum\limits_{M/2^{k+1}<|\alpha|\leq M/2^k}\sum_{|y|\leq M/2|\alpha|}   \tau(y)
\ll \sum_{k\geq 1} \frac{M^2}{2^{2k}} \cdot 2^k k\ll M^2.
\end{equation*}

Substituting the latter into \eqref{eq:summand u=1 sum} it reads (this is the summand in
\eqref{eq:sum close pairs alph,beta} corresponding to $u=1$)
\begin{equation}
\label{eq:summand u=1 sum int + error}
\begin{split}
&\#\bigg\{ \alpha, \beta \in \mathbb Z[i]:\  |\beta| \leq \frac{\sqrt{N}}{|\alpha|}, |\beta-\overline{\beta}|\leq \frac{M}{|\alpha|}
, \ \  (\beta,\overline{\beta})=1 \bigg\}
\\&=  4\sum_{y:\  1\leq y \leq M/2 } \frac{\varphi(2y)}{2y}  \sum_{\substack{\alpha\in \mathbb Z[i] \\ |\alpha|\leq M/2y}}   \left(\frac N{ |\alpha|^2}-y^2 \right)^{1/2} + O(M^2).
\end{split}
\end{equation}
Now define $R(t):=\sum_{\alpha\in \mathbb Z[i], \ |\alpha|\leq T} 1 = \pi T^2+O(T)$, and use summation by parts to evaluate 
the inner sum on the r.h.s. of \eqref{eq:summand u=1 sum int + error}. We have
\begin{equation}
\label{eq:inn intg 1 + error}
\begin{split}
&\sum_{\substack{\alpha\in \mathbb Z[i] \\ |\alpha|\leq M/2y}}   \left(\frac N{ |\alpha|^2}-y^2 \right)^{1/2}=
\int_1^{M/2y}  \left(\frac N{ t^2}-y^2 \right)^{1/2} dR(t) \\&= 
\int_1^{M/2y}  \left(\frac N{ t^2}-y^2 \right)^{1/2} d(\pi t^{2}+O(t)) \\&
=2\pi N^{1/2}\int_1^{M/2y}  \left( 1-\frac{y^2t^{2}}{N} \right)^{1/2}dt + O (N^{1/2}\log(M/2y)),
\end{split}
\end{equation}
where the above formal treatment in the last equality in \eqref{eq:inn intg 1 + error}
hides applying summation by parts followed by integration by parts in ``opposite direction", 
noting that the boundary terms cancel each other upon the
sequential applications of the summation by parts, and the relevant summands are of the same sign (so that we can differentiate the error term). 
To evaluate the integral on the r.h.s. of \eqref{eq:inn intg 1 + error}
we transform the variables $Mv=2yt, $ so that 
\begin{equation*}
\int_{1}^{M/2y}  \left( 1-\frac{y^2t^{2}}{N} \right)^{1/2}dt = \frac{M}{2y}\int\limits_{2y/M}^{1}\left( 1-\frac{M^{2}}{4N}u^{2} \right)^{1/2}dv
=\frac{M}{2y}\left(I\left(\frac{M^{2}}{4N}\right)+O\left(\frac{y}{M}\right)\right),
\end{equation*}
upon extending the range of the integral and recalling the definition \eqref{eq:I(c) def} of $I(c)$.
Substituting the latter into \eqref{eq:inn intg 1 + error} yields
\begin{equation}
\label{eq:inn intg I(c) + error}
\sum_{\substack{\alpha\in \mathbb Z[i] \\ |\alpha|\leq M/2y}}   \left(\frac N{ |\alpha|^2}-y^2 \right)^{1/2} 
= \frac{\pi^{2} N^{1/2} M}{4y} I\left(\frac{M^{2}}{4N}\right) +O\left(N^{1/2} \left(\log \left(\frac{M}{2y}\right)+1\right)\right).
\end{equation}
We then find that the sum of the error terms on the r.h.s. of \eqref{eq:inn intg I(c) + error}
along the range of summation of \eqref{eq:summand u=1 sum int + error}
is bounded by
\begin{equation}
\label{eq:sum errors<<M^2logM+N^1/2M}
\ll N^{1/2}  \sum_{y:\  1\leq y \leq M/2} \log \left(\frac{M}{2y}\right) + N^{1/2}M \ll N^{1/2}  M
\end{equation}
by comparing the summation in \eqref{eq:sum errors<<M^2logM+N^1/2M} to the corresponding integral.

Now we substitute the estimate \eqref{eq:inn intg I(c) + error} into \eqref{eq:summand u=1 sum int + error},
and use the bound \eqref{eq:sum errors<<M^2logM+N^1/2M} for the relevant summation of the error terms to obtain
\begin{equation}
\label{eq:summand u=1=N^1/2Msum+errs}
\begin{split}
&\#\left\{ \alpha, \beta \in \mathbb Z[i]:\  |\beta| \leq \frac{\sqrt{N}}{|\alpha|}, |\beta-\overline{\beta}|\leq \frac{M}{|\alpha|}, \ \  (\beta,\overline{\beta})=1 \right\}
\\&=  \pi^{2} N^{1/2}M \cdot I\left(\frac{M^{2}}{4N}\right)\cdot \sum\limits_{y:\  1\leq y \leq M/2 } \frac{\varphi(2y)}{2y^{2}} +O(N^{1/2}M),
\end{split}
\end{equation}
where the error term $O(N^{1/2}M)$ also encapsulates $O(M^{2})$ from \eqref{eq:summand u=1 sum int + error}, as $M\le 2\sqrt{N}$.

For the main term on the r.h.s. of \eqref{eq:summand u=1=N^1/2Msum+errs} we need to determine
\begin{equation}
\label{eq:sum phi(2y)/2y^2 sums}
 \begin{split}
 \sum_{y:\  1\leq y \leq M/2 } \frac{\varphi(2y)}{2y^2} &=  \frac 12 \sum_{y:\  1\leq y \leq M/2 } \frac 1y  \sum_{\substack{ d|y \\ d \ \text{odd}}}
 \frac{\mu(d)} d = \frac 12 \sum_{\substack{ d\leq M/2 \\ d \ \text{odd}}}
 \frac{\mu(d)} d   \sum_{y  \leq M/2,\ d|y } \frac 1y    \\
 &= \frac 12 \sum_{\substack{ d\leq M/2 \\ d \ \text{odd}}} \frac{\mu(d)} {d^2}   \sum_{m \leq M/2d } \frac 1{ m}
 \\&= \frac 12 \sum_{\substack{ d\leq M/2 \\ d \ \text{odd}}} \frac{\mu(d)} {d^2}   \left( \log (M/2d) +\gamma + O(d/M) \right)\\
 &=  \frac 12 \sum_{\substack{ d\leq M/2 \\ d \ \text{odd}}} \frac{\mu(d)} {d^2}     \log (M/2d) +\   \frac \gamma2 \sum_{\substack{ d\leq M/2 \\ d \ \text{odd}}} \frac{\mu(d)} {d^2}    + O\left(   \frac {\log M}M     \right),
 \end{split}
\end{equation}
writing $y=dm$. Now
\begin{equation}
\label{eq:sum mu(d)/d^2=8/pi^2+O(1/M)}
 \sum_{\substack{ d\leq M/2 \\ d \ \text{odd}}} \frac{\mu(d)} {d^2} = \sum_{\substack{  d \ \text{odd},\geq 1}} \frac{\mu(d)} {d^2}  +O\left(  \sum_{ d> M/2 } \frac{1} {d^2}  \right) = \frac 8{\pi^2} +O(1/M).
\end{equation}
Also
\begin{equation}
\label{eq:sum mu(d)log(d)/d^2=8/Pi^2c}
\begin{split}
 \sum_{\substack{   d \ \text{odd}}} \frac{\mu(d)\log d} {d^2} &=
 \sum_{\substack{  d \ \text{odd}}} \frac{\mu(d)} {d^2}  \sum_{p|d} \log p =
 \sum_{p \ \text{odd}} \log p  \sum_{\substack{   d \ \text{odd}\\ p|d}} \frac{\mu(d)} {d^2}
 = - \frac 8{\pi^2} \sum_{p \ \text{odd}} \frac{ \log p}{p^2-1} .
 \end{split}
\end{equation}
Combining the estimates \eqref{eq:sum mu(d)/d^2=8/pi^2+O(1/M)} and \eqref{eq:sum mu(d)log(d)/d^2=8/Pi^2c} and inserting them into
\eqref{eq:sum phi(2y)/2y^2 sums} gives
\begin{equation}
\label{eq:sum phi(2y)/2y^2 asymp}
\begin{split}
 \sum_{y:\  1\leq y \leq M/2 } \frac{\varphi(2y)}{2y^2} &= \frac {4}{\pi^2} \left( \log M/2 + \gamma +
 \sum_{p \ \text{odd}} \frac{ \log p}{p^2-1}  \right)+ O\left(   \frac {\log M}M     \right) \\&= \frac {4}{\pi^2}\log{M} \cdot \left(1+O\left(\frac{1}{\log{M}}\right)\right),
\end{split}
\end{equation}
so that \eqref{eq:summand u=1=N^1/2Msum+errs} is
\begin{equation}
\label{eq:summand u=1 asymp}
\begin{split}
&\#\bigg\{ \alpha, \beta \in \mathbb Z[i]:\  |\beta| \leq \frac{\sqrt{N}}{|\alpha|}, |\beta-\overline{\beta}|\leq \frac{M}{|\alpha|}\bigg\}
\\&=  4N^{1/2}M\log{M} \cdot I\left(\frac{M^{2}}{4N}\right)\cdot \left(1+O\left(\frac{1}{\log{M}}\right)\right),
\end{split}
\end{equation}
where the error term in the latter estimate also encapsulates the one in \eqref{eq:summand u=1=N^1/2Msum+errs}.
The estimate \eqref{eq:summand u=1 asymp} means that the term in the sum on the r.h.s. of \eqref{eq:sum close pairs alph,beta}
corresponding to $u=1$ contributes $\frac{1}{4}$ of what is claimed in
the statement \eqref{eq:not BR seq asympt Andrew} of Theorem \ref{thm:not BR seq asympt Andrew}.

We claim that the contribution of each of the other three terms in \eqref{eq:sum close pairs alph,beta} is also
given by the r.h.s. of \eqref{eq:summand u=1 asymp}. While the proofs
are very similar we highlight the differences for the convenience of the reader.
The $u=-1$ term yields the conditions
\[
 |x|\leq M/2|\alpha|\ \ \text{and} \ \ y^2\leq N/ |\alpha|^2-x^2, \ \text{with} \ \ (x,y)=1  \ \text{and} \ x+y \ \text{odd};
\]
that is, the roles of $x$ and $y$ are reversed as compared to \eqref{eq:cond x,y u=1}; one then gets the same estimate.
If $u=i$ then $\beta-u\overline{\beta}=(1-i)(x-y)$, so let $y=x+\Delta$ so that $x^2+y^2\leq T^2$ becomes
$(2x+ \Delta)^2\leq 2T^2-\Delta^2$. Therefore we have the conditions, for $X=(2N/|\alpha|^2-\Delta^2)^{1/2}$
\[
\begin{split}
| \Delta|\leq M/\sqrt{2} |\alpha|\ \ &\text{and} \ \frac{-X- \Delta}2 \leq  x \leq
\frac{X- \Delta}2\\&\text{with} \ \ (x, \Delta)=1  \ \text{and} \  \Delta \ \text{odd}.
\end{split}
\]
We now have a slightly different calculation from before; we will note the differences:\ Again $x$ runs through an interval of length $X$, and so the number of such $x$ is $$\frac{\varphi(\Delta)}{\Delta} \cdot X+O(\tau(\Delta)).$$ Running through the calculation we get a main term of
\[
\sum_{\substack{ \Delta\leq M/\sqrt{2} \\  \Delta \ \text{odd}}}  \frac{\varphi(\Delta)}{\Delta^2} \cdot  \frac{\pi N^{1/2} M}{y}
\]
with the same error terms. An analogous calculation reveals that
\[
\sum_{\substack{ \Delta\leq M/\sqrt{2} \\  \Delta \ \text{odd}}}  \frac{\varphi(\Delta)}{\Delta^2}
= \frac 4{\pi^2} \left( \log \sqrt{2} M + \gamma + \sum_{p \ \text{odd}} \frac{ \log p}{p^2-1}  \right) +
O\left(   \frac {\log M}M     \right) .
\]
A similar calculation ensues for $u=-i$. Therefore, as mentioned above and similar to \eqref{eq:summand u=1 asymp}, each summand of \eqref{eq:sum close pairs alph,beta} contribute a quarter of the total claimed \eqref{eq:not BR seq asympt Andrew}, and the result follows.
\end{proof}


\subsection{Automorphisms of pairs of close lattice points}
\label{sec:NumAutoClasses}
We observe that if $ (\alpha,\beta,u)\in S(n,M)$ corresponds to $(a,b)\in \Ec_{n}$ then, taking conjugates,
\begin{equation}
\label{eq:aut conj S(n,M)}
(\overline{\alpha},\overline{\beta},\overline{u})\in S(n,M)
\end{equation}
corresponds to $(\overline{a},\overline{b})\in \Ec_{n}$. More interestingly, given $(\alpha,\beta,u)\in S(n,M)$ we see that
\begin{equation}
\label{eq:Aalbet def}
A(\alpha,\beta,u):= \{  (\alpha',\beta w,uw^2):\ \alpha'\in \mathbb Z[i] \ \text{with} \ |\alpha'|=|\alpha|\ \text{and } \ w\in \CU \}
\end{equation}
is a subset of $S(n,M)$. Hence we can partition $S(n,M)$ up into   sets
\begin{equation}
\label{eq:A*albet def}
A^*(\alpha,\beta,u) := A(\alpha,\beta,u) \cup A (\overline{\alpha},\overline{\beta},\overline{u}).
\end{equation}
How often $S(n,m)$ is equal to some unique $A^*(\alpha,\beta,u)$? We can re-formulate this question by letting
\[
\mathcal A(n,M):=\{ A^*(\alpha,\beta,u):\  (\alpha,\beta,u)\in S(n,M) \}
\]
and asking how often $|\mathcal A(n,M)|>1$.

\begin{theorem} \label{NumAutoClasses} Suppose that $M\leq 2\sqrt{N}$ and let
\begin{equation}
\label{eq:c=M^2/4N}
c=\frac{M^2}{4N}.
\end{equation}
\begin{enumerate}

\item
The number of distinct sets $A^*(\alpha,\beta,u)$ with $ |\alpha| \cdot |\beta| \leq \sqrt{N} $ is asymptotic to
\[
\sum_{n\leq N} |\mathcal A(n,M)| =
\frac{\kappa'}{2} \cdot I(c) \cdot MN^{1/2} \cdot
( (2\log M)^{1/2} +O(1)).
\]

\item
The number of pairs of distinct close-by pairs is
\begin{equation}
\label{eq:pairs A << N^1/2M^4/3log^6}
\sum_{n\leq N} \binom{|\mathcal A(n,M)|}2  \ll  N^{1/3}M^{4/3} (\log N)^6 +N^{1/2}  (\log N)^3.
\end{equation}

\end{enumerate}
\end{theorem}

\begin{proof}[Proof of Theorem \ref{thm:G* asymp} assuming Theorem \ref{NumAutoClasses}]

If $m$ is a non-negative integer then the characteristic function
\begin{equation*}
\mathds{1}_{\geq 1}(m) = \begin{cases}
1 &m\ge 1\\
0 &m=0
\end{cases}
\end{equation*}
satisfies
\begin{equation}
\label{eq:1_>=1(m)=m+O(m choose 2)}
\mathds{1}_{\geq 1}(m) = m+O\left( \binom m2 \right)
\end{equation}
and
\begin{equation}
\label{eq:1_>=1(m)<=m}
\mathds{1}_{\geq 1}(m) \leq m.
\end{equation}
Now for a given $n\in S$, and $M>0$, there exists a pair $\lambda,\lambda'\in \Ec_{n}$ with $0<\|\lambda-\lambda'\|<M$,
if and only if $|\mathcal A(n,M)| \ge 1$. Hence, bearing in mind the definition \eqref{eq:G^*def}
of $\Gc^*(N;M)$, we have
\begin{equation}
\label{eq:Gc=sum 1_>=1Ac}
\Gc^*(N;M) = \sum\limits_{n \le N} \mathds{1}_{\geq 1}(|\mathcal A(n,M)|).
\end{equation}
Substituting \eqref{eq:1_>=1(m)=m+O(m choose 2)} into \eqref{eq:Gc=sum 1_>=1Ac} we 
obtain
\begin{equation*}
\begin{split}
\Gc^*(N;M) &= \sum\limits_{n \le N} |\mathcal A(n,M)| + O\left( \sum_{n\leq N} \binom{|\mathcal A(n,M)|}2  \right)
\\&=\frac{\kappa'}{2}\cdot I(c)\cdot MN^{1/2} ( (2\log M)^{1/2} +O(1)) \\&+ O\left(N^{1/3}M^{4/3} (\log N)^6 + N^{1/2}(\log{N})^{3}\right),
\end{split}
\end{equation*}
by both parts of Theorem \ref{NumAutoClasses}. The first part of Theorem \ref{thm:G* asymp}
finally follows from substituting \ref{eq:I(c)=I(0)+O(c)} into the latter estimate, recalling that here we
assumed \eqref{eq:assump M<=N^1/2/logN^17}, and noting
$$N^{1/3}M^{4/3}=\frac{MN^{1/2}}{(\sqrt{N}/M)^{1/3}}.$$

To prove the second part of Theorem \ref{thm:G* asymp} we substitute \eqref{eq:1_>=1(m)<=m} into \eqref{eq:Gc=sum 1_>=1Ac}
to yield
\begin{equation*}
\Gc^*(N;M) \le \sum\limits_{n \le N} |\mathcal A(n,M)| = \frac{\kappa'}{2}\cdot I(c)\cdot  MN^{1/2} ( (2\log M)^{1/2} +O(1)).
\end{equation*}
The desired result follows at once from the fact that $I(c)$ is decreasing on $[0,1]$, so that for every $c\in [0,1]$,
$I(c)\le I(0) = \frac{4}{\pi}$.

\end{proof}

\subsection{Proof of Theorem \ref{NumAutoClasses}, part I}

\begin{proof}
Write  $|\alpha|^2=a$ and $\beta=p+iq$; evidently $a\in S$.
In \eqref{eq:aut conj S(n,M)}, \eqref{eq:Aalbet def}, \eqref{eq:A*albet def}, we see that the set $\Ac(n,M)$ is designed to take care of an automorphism group (of pairs of close-by lattice points on the circle of radius  $\sqrt{n}$) of order $8$. Therefore
\begin{equation}
\label{eq:sum A(n,M) over U, a}
\begin{split}
&\sum_{n\leq N} |\mathcal A(n,M)| \\=\frac 18 \sum_{u\in \CU} \sum_{a\in S}    \# \bigg\{      (\beta,u)\in \mathbb Z[i]\times \CU:\
&|\beta| \leq \sqrt{N /a},\ \ \  |\beta-u\overline{\beta}|\leq \frac{M}{\sqrt{a}},\ (\beta,\overline \beta)=1  \bigg \},
\end{split}
\end{equation}
where, as before, $S$ denotes the set of integers that are sums of two squares. Let $\beta=x+iy$ so that $x+y$ is odd, and $(x,y)=1$.

In the case $u=1$ we have $|y|\leq M/ 2\sqrt{a} $ and then $x^2\leq N/a-y^2$. We will proceed analogously to
the proof of Theorem \ref{thm:not BR seq asympt Andrew}, but now we have, thanks to Landau \eqref{eq:Landau M/sqrt(logM)},
\begin{equation}
\label{eq:S(t) asympt}
S(t):= \sum_{n\in S,\ n\leq t} 1 = \kappa_{LR} \frac t{(\log t)^{1/2}} \left( 1 + O\left( \frac 1{\log t}\right)\right) ,
\end{equation}
so that the term on the r.h.s. of the summation in \eqref{eq:sum A(n,M) over U, a} corresponding to $u=1$ contributes
\begin{equation}
\label{eq:term u=1 sum ind}
\begin{split}
&\sum_{a\in S}    \# \bigg\{      (\beta,1)\in \mathbb Z[i]\times \CU:\
|\beta| \leq \sqrt{N /a},\ \ \  |\beta-\overline{\beta}|\leq \frac{M}{\sqrt{a}} \ \text{and} \ (\beta,\overline \beta)=1  \bigg \}
\\&=4 \sum_{y:\  1\leq y \leq M/2 } \frac{\varphi(2y)}{2y}  \sum_{\substack{a\leq (M/2y)^2 \\ a\in S}}
\left(\frac N{ a}-y^2 \right)^{1/2} + O\left( \frac{M^2}{(\log M)^{1/2}}  \right)
\\&= 4 \sum_{y:\  1\leq y \leq M/2-1} \frac{\varphi(2y)}{2y}  \sum_{\substack{a\leq (M/2y)^2 \\ a\in S}}
\left(\frac N{ a}-y^2 \right)^{1/2} +O(N)+ O\left( \frac{M^2}{(\log M)^{1/2}}  \right);
\end{split}
\end{equation}
here, to avoid vanishing denominator later, we separated the contribution of $y\le\frac{M}{2}-1$, using the trivial bound $O(N)$
to each of the summands with $y>\frac{M}{2}-1$, whose number is $O(1)$.
In this case the inner sum is more complicated as compared to \eqref{eq:inn intg 1 + error}:
we formally write (again hiding summation by parts followed by integration by parts in ``opposite direction", 
much in the spirit of \eqref{eq:inn intg 1 + error})
\begin{equation*}
\sum_{\substack{a\leq (M/2y)^2 \\ a\in S}}
 \left(\frac N{ a}-y^2 \right)^{1/2} = \int\limits_{1}^{(M/2y)^{2}} \left(\frac N{ t}-y^2 \right)^{1/2} dS(t)
\end{equation*}
via \eqref{eq:S(t) asympt} to yield
\begin{equation}
\label{eq:inner sum int parts ind}
\begin{split}
\sum_{\substack{a\leq (M/2y)^2 \\ a\in S}}
 &\left(\frac N{ a}-y^2 \right)^{1/2}\\=
\kappa_{LR} \int_1^{(M/2y)^2}  \left(\frac N{ t }-y^2 \right)^{1/2} &\frac{dt }{(\log t)^{1/2}}
+ O\left( \frac{MN^{1/2}}{ y(\log M/2y)^{3/2}}  \right)
\\=\frac{\kappa_{LR} MN^{1/2}}{y(2\log (M /2y))^{1/2} }  \int_{2y/M}^{1}   &\left(  1-\frac{M^2}{4N}v^2  \right)^{1/2}dv+ O\left( \frac{MN^{1/2}}{ y(\log M/2y)^{3/2}}   \right),
\end{split}
\end{equation}
and letting $t=(Mv/2y)^2$. 
Extending the range of the latter integral to $0$, we see that it equals
\begin{equation*}
\int_{2y/M}^{1}   \left(  1-\frac{M^2}{4N}v^2  \right)^{1/2}dv = \frac{\pi}{4}I(c) + O\left(\frac{y}{M}\right),
\end{equation*}
by the definition \eqref{eq:I(c) def} of $I(c)$, \eqref{eq:c=M^2/4N}, and the boundedness of the integrand.
We then have upon substituting the latter estimate into
\eqref{eq:inner sum int parts ind}, and then into \eqref{eq:term u=1 sum ind},
that the $u=1$ term in \eqref{eq:sum A(n,M) over U, a} contributes to the sum
\begin{equation}
\label{eq:sum A(n,M) over U, a as phi(2y)/2y^2}
\begin{split}
&\sum_{a\in S}    \# \bigg\{      (\beta,1)\in \mathbb Z[i]\times \CU:\
|\beta| \leq \sqrt{N /a},\ \ \  |\beta-\overline{\beta}|\leq \frac{M}{\sqrt{a}} \ \text{and} \ (\beta,\overline \beta)=1  \bigg \}
\\&= \pi\kappa_{LR} MN^{1/2} \cdot I(c) \sum_{y:\  1\leq y \leq M/2-1 } \frac{\varphi(2y)}{2y^2} \frac{1}{ (2\log (M /2y))^{1/2} } + E,
\end{split}
\end{equation}
where the error term $E=E(N,M)$ is bounded by
\begin{equation}
\label{eq:sum A(n,M) over U, a as phi(2y)/2y^2 err}
\begin{split}
|E| &\ll  \sum_{y:\  1\leq y \leq M/2 } \frac{\varphi(2y)}{2y}
\left(  \frac{MN^{1/2}}{ y(\log M/2y)^{3/2}}  +
\frac{  N^{1/2}}{(\log (M /2y))^{1/2} }   \right)
+  \frac{M^2}{(\log M)^{1/2}}   \\&\ll MN^{1/2} .
\end{split}
\end{equation}
To evaluate the main term we reuse our estimate
$$ P(t):=\sum_{y\leq t} \frac{\varphi(2y)}{2y^2}  = \frac {4}{\pi^2} \left(\log t+C+O\left(\frac{\log{t}}{t}\right)\right)$$ for some constant $C$ (cf. \eqref{eq:sum phi(2y)/2y^2 asymp}), and plan to use summation by parts followed by integration by parts,
again in the spirit of \eqref{eq:inn intg 1 + error}. 
And so we get, formally manipulating (assume for simplicity that $M/2\in \Z$, otherwise further restrict the range of integration),
\begin{equation}
\label{eq:clever sum parts phi(2y)/2y^2 1/log}
\begin{split}
\sum_{y:\  1\leq y \leq M/2-1 } \frac{\varphi(2y)}{2y^2} \frac{1}{ (2\log (M /2y))^{1/2} } &=
\int_{1}^{M/2-1 } \frac{dP(y)}{(2\log (M /2y))^{1/2}} \\&= \frac{4}{\pi^2}  \int\limits_{1}^{M/2-1 }\frac{d(\log y+C+O((\log y)/y))}{(2\log (M /2y))^{1/2}}
\\&= \frac{4}{\pi^2}  \int\limits_{1}^{M/2-1 } \frac{dy}{y(2\log (M /2y))^{1/2}} + E',
\end{split}
\end{equation}
where the error term is
\begin{equation}
\label{eq:error sum phi(2y)/2y^2 log}
|E'| \le \frac{\log y}{y(\log(M/2y))^{1/2}}\bigg|_{y=1}^{y=M/2-1} +  \int_{1}^{M/2-1 }\frac{\log{y}dy}{y^{2}(\log(M/2y))^{3/2}} \ll
\frac{\log{M}}{M^{1/2}}+ 1,
\end{equation}
by changing the variables $t=\frac{M}{2y}$ and separating the contribution of the range $y\in [1,\epsilon M]$ and $[\epsilon M , M/2]$.

We may then evaluate the latter integral in \eqref{eq:clever sum parts phi(2y)/2y^2 1/log} explicitly to be
\begin{equation*}
\int\limits_{1}^{M/2-1 } \frac{dy}{y(2\log (M /2y))^{1/2}} = (2\log(M))^{1/2}+O(1),
\end{equation*}
and, with the help of \eqref{eq:error sum phi(2y)/2y^2 log}, obtain
\begin{equation*}
\sum_{y:\  1\leq y \leq M/2-1 } \frac{\varphi(2y)}{2y^2} \frac{1}{ (2\log (M /2y))^{1/2} } = \frac{4}{\pi^{2}}(2\log{M}+O(1))^{1/2}.
\end{equation*}
Substituting the latter estimate into \eqref{eq:sum A(n,M) over U, a as phi(2y)/2y^2} and bearing in mind
\eqref{eq:sum A(n,M) over U, a as phi(2y)/2y^2 err} we finally obtain
\begin{equation*}
\begin{split}
&\sum_{a\in S}    \# \bigg\{      (\beta,1)\in \mathbb Z[i]\times \CU:\
|\beta| \leq \sqrt{N /a},\ \ \  |\beta-\overline{\beta}|\leq \frac{M}{\sqrt{a}} \ \text{and} \ (\beta,\overline \beta)=1  \bigg \}
\\&= \frac{4}{\pi}\kappa_{LR} MN^{1/2}\cdot I(c) \cdot (2\log{M}+O(1))^{1/2} ,
\end{split}
\end{equation*}
which contributes a quarter in \eqref{eq:sum A(n,M) over U, a} of what is stated in part I of Theorem \ref{NumAutoClasses}.
We get a similar quantity for $u=-1$; and by suitably modifying the proof, we get the same quantity for $u=i$ and $u=-i$, modifying the proof much like we did in Theorem \ref{thm:not BR seq asympt Andrew}.
\end{proof}

\subsection{Many pairs - proof of Theorem \ref{NumAutoClasses}, part II}

\begin{lemma} \label{Simple count}
For any $\theta$ we have, uniformly,
\[
\begin{split}
\# &\{ x+iy \in \mathbb Z[i]:\ |x+iy|\leq N, \ (x,y)=1\ \& \ | \text{\rm arg}(x+iy)-\theta | <\epsilon \}
\ll 1+ \epsilon N^2.
\end{split}
\]
In fact if $\epsilon\leq 1/2N^2$ there is no more than one solution.
\end{lemma}

The proof of Lemma \ref{Simple count} will be given immediately after the proof of Theorem \ref{NumAutoClasses}, part II.

\begin{proof} [Proof of Theorem \ref{NumAutoClasses}, part II assuming Lemma \ref{Simple count}] We treat separately those
$n\asymp N$, which are either a square nor twice a square. These contribute at most
\[
\sum_{m^2 \asymp N} r_2(m^2)^2+\sum_{2m^2 \asymp N}r_2(2m^2)^2 \ll N^{1/2} (\log N)^3.
\]

Now suppose we have two pairs $a_1,b_1$ and $a_2,b_2$ not belonging to the same class $A^{*}(\alpha,\beta,u)$, but with the same $n\asymp N$, which is neither a square nor twice a square.
We use the identification  between the close-by pairs and triples $(\alpha,\beta,u)\in S(n,M)$
as in the beginning of section \ref{sec:ref aut}, so that a couple of close-by pairs yields two triples $(\alpha_{j},\beta_{j},u_{j}) \in S(n,M)$,
$j=1,2$.
It is then possible to write
\begin{equation}
\label{eq:betaj=gam del tj}
\beta_1= \gamma \delta \theta_1 \ \ \text{and} \ \  \beta_2= \gamma \overline{\delta} \theta_2
\end{equation}
where $(\beta_1, \beta_2)=(\gamma)$ and $(\beta_1/\gamma,\overline{ \beta_2/\gamma})=(\delta)$,
and the norms of $\theta_1$ and $\theta_2$ are coprime. Since $|\alpha_{1}\beta_{1}| = |\alpha_{2}\beta_{2}| = \sqrt{n}$, we have
\begin{equation}
\label{eq:alphaj=tjvj}
\alpha_1= t_2v_1 \ \ \text{and} \ \  \alpha_2= t_1v_2,
\end{equation}
where
\begin{equation}
\label{eq:tj=|tj|,|v1|=|v2|}
|t_1|=| \theta_1|,\ |t_2|=| \theta_2| \text{ and } |v_1|=|v_2|.
\end{equation}
Hence we must have
\begin{equation}
\label{eq:gam del t1 t2 v1<>N^1/2}
|\gamma \delta \theta_1  \theta_2 v_1| = |\beta_{1}\alpha_{1}| = \sqrt{n}\asymp \sqrt{N},
\end{equation}
and we in addition have
\begin{equation}
\label{eq:theta1v1,theta2v1<=M}
\ |  \theta_1  v_1| = |\alpha_{2}| = \frac{M}{|\beta-u\overline{\beta}|} \le M ,\, \text{and } |\theta_2  v_1|=|\alpha_{1}| \le M
\end{equation}
in a similar fashion.
Substituting the estimates \eqref{eq:theta1v1,theta2v1<=M} into \eqref{eq:gam del t1 t2 v1<>N^1/2} we conclude
that
\begin{equation}
\label{eq:gamdelt1,gamdelt1>>N^1/2/M}
|\gamma \delta \theta_1 |,\ |\gamma \delta \theta_2|  \gg  \frac{\sqrt{N}}{M}  .
\end{equation}

Next we use the condition that
\begin{equation*}
|\alpha_{j}|\cdot |\beta_{j}-u_{j}\overline{\beta}|\le M.
\end{equation*}
Since $|\alpha_{j}|\cdot |\beta_{j}| = \sqrt{n}\asymp  \sqrt{N}$, this yields
\begin{equation*}
|1-\exp(-2i\cdot \text{arg}(\beta_{j}))| = \left|1-u_{j}\frac{\overline{\beta_{j}}}{\beta_{j}}\right|\ll \frac{M}{\sqrt{N}}.
\end{equation*}
Recalling that $u_{j}$ are units, this implies that
$2\, \text{arg}(\beta_{j})$ are small modulo $\frac{\pi}{2}$, or, more precisely
$$|\text{arg}(\beta_j)| \ll \frac{M}{\sqrt{N}} \mod \frac{\pi}{4}.$$
Therefore, bearing in mind \eqref{eq:betaj=gam del tj}, we have
\begin{align*}
 \arg(\theta_1)  = - \arg \gamma - \arg \delta  &\pmod{\pi/4} \ \  +O( M/\sqrt{N}) \\
 \arg(\theta_2) = - \arg \gamma + \arg \delta  &\pmod{\pi/4} \ \ +O(  M/\sqrt{N})
\end{align*}
We have two linearly independent equations in four unknowns. This allows us to determine
$\arg \gamma, \ \arg \delta$ in terms of $\arg(\theta_1), \ \arg(\theta_2)$ mod $\pi/4$, with error $O(M/\sqrt{N})$:
\begin{equation}
\label{eq:2arg(gam)+arg(t1)+arg(t2) small}
\begin{split}
2\arg \gamma +(\arg(\theta_1)+\arg(\theta_2)) &\pmod{\pi/4} \ \ = O( M/\sqrt{N}) \\
2\arg \delta + (\arg(\theta_1)-\arg(\theta_2)) &\pmod{\pi/4} \ \ = O( M/\sqrt{N}).
\end{split}
\end{equation}

\vspace{2mm}

The inequalities \eqref{eq:2arg(gam)+arg(t1)+arg(t2) small} imply that there exist $w,w'\in \pm \{ 1,i,1+i,1-i\}$ such that
\begin{equation}
\label{eq:arg small}
\arg(w\delta^2\theta_1\overline{\theta_2}),\arg(w'\gamma^2\theta_1\theta_2) = O\left( \frac{M}{\sqrt{N}}\right).
\end{equation}
We argue that neither of the numbers on the l.h.s. of \eqref{eq:arg small} can vanish precisely:\
For if we suppose that $\arg(w\delta^2\theta_1\overline{\theta_2}) =0$,  then, as  $\theta_1$ and $\theta_2$ are coprime, we can write
$\theta_j=r_j\phi_j ^2\omega_j$ for $j=1,2$ with $\delta = \overline{\phi_1}\phi_2 $, where each $\omega_i$ divides $\omega$, and the $r_j$ are integers.  Now $(r_j \omega_j)$ divides $(\theta_j,\overline{\theta_j})$, which divides $(\beta_j,\overline{\beta_j})=1$, and so
$r_j \omega_j$ is a unit. Moreover $$\phi_1\phi_2|(\delta\theta_{1},\overline{\delta}\theta_{2})=
(\beta_1/\gamma, \beta_2/\gamma)=(1),$$ and so $\phi_1=\phi_2=1$, and therefore
$\delta=\overline{\phi_1}\phi_2=1$.   Therefore  $\theta_1,\theta_2$ are units.
We may then select $\gamma$ so that $\theta_2=1$,
and so $\beta_2=w \beta_1$ for some unit $w$, and $|\alpha_1|=| \alpha_2|$. Therefore
$(\alpha_2,\beta_2,u_2)\in A(\alpha_1,\beta_1,u_1)$ in contradiction to our assumption that these triples belong to different
classes  $A^{*}(\alpha,\beta,u)$. Similarly if $\arg(w'\gamma^2\theta_1\theta_2) =0$ then we take the conjugate of $(\alpha_2,\beta_2)$, so that the roles of $\gamma$ and $\delta$ are exchanged, and so
$(\alpha_2,\beta_2,u_2)\in  A (\overline{\alpha_1},\overline{\beta_1},\overline{u_1})$. In either case we get a contradiction
to our assumption that $(\alpha_{j},\beta_{j},u_{j})$, $j=1,2$ belong to different classes $A^{*}(\alpha,\beta,u)$.

Therefore we may assume $$0\ne \arg(w\delta^2\theta_1\overline{\theta_2}) = O\left( \frac{M}{\sqrt{N}}\right) ,$$ and so
\[
1\leq |\text{Im}(w\delta^2\theta_1\overline{\theta_2})| \asymp |\arg(w\delta^2\theta_1\overline{\theta_2})| \cdot | \delta^2\theta_1\overline{\theta_2}|
\ll | \delta^2\theta_1\overline{\theta_2}| \cdot \frac{ M}{\sqrt{N}},
\]
and so we deduce that
\begin{equation}
\label{eq:gam^2t1t2,del^2t1t2>>N^1/2/M}
|\gamma^2\theta_1\theta_2|  \gg \frac{\sqrt{N}}{M}, \ \ \text{and  } \ \ |\delta^2\theta_1\overline{\theta_2}|  \gg  \frac{\sqrt{N}}{M}.
\end{equation}

\vspace{2mm}

To summarize all the above, a pair of triples $(\alpha_{j},\beta_{j},u_{j})$ that satisfy the conditions above and belong to different classes $A^{*}(\cdot,\cdot,\cdot)$ determines an $8$-tuple $(\gamma,\delta,\theta_{1},\theta_{2},t_{1},t_{2},v_{1},v_{2})\in \Z[i]^{8}$
that satisfies \eqref{eq:tj=|tj|,|v1|=|v2|}, \eqref{eq:gam del t1 t2 v1<>N^1/2}, \eqref{eq:theta1v1,theta2v1<=M}, \eqref{eq:gamdelt1,gamdelt1>>N^1/2/M} and \eqref{eq:gam^2t1t2,del^2t1t2>>N^1/2/M}. Conversely, such a $8$-tuple corresponds to
a unique $(\alpha_{j},\beta_{j})$ via \eqref{eq:betaj=gam del tj} and \eqref{eq:alphaj=tjvj},
hence, for our purposes it is sufficient to bound their number (bearing in mind that the group of units is finite).
Let $C,D,T_1,T_2,V$ be powers of $2$ that are greater than and closest to
$|\gamma|,\,|\delta|,\,|\theta_{1}|=|t_{1}|,|\theta_{2}|=|t_{2}|$, and $|v_{1}|=|v_{2}|$ respectively;
hence their product is $ CDT_{1}T_{2}V\asymp \sqrt{N}$ with $T_1V, T_2V \ll M$, and
\begin{equation} \label{LowerBds}
CDT_1,\ CDT_2,\ C^2T_1T_2,\ D^2T_1T_2 \gg  \sqrt{N}/M.
\end{equation}
We note that $\sum_{n\leq N} r_2(n)^2\asymp N \log N$.
First select an integer $v \asymp V^2$ and any $v_1,\, v_2$ with $|v_{1}|^{2}=|v_{2}|^{2}=v$, so the total number of possible choices for $v_{1},v_{2}$ is $$\leq \sum_{v\ll V^2} r_2(v)^2\ll V^2 \log V\ll V^2 \log M.$$

\vspace{2mm}

Let $A$ be the largest of $C,D,T_1,T_2$, and $B$ be the second largest. By \eqref{LowerBds}, we have
\begin{equation}
\label{eq:AB>>,B>>}
AB\gg \left(\frac{\sqrt{N}}{M}\right)^{2/3}\text{ and }\; B\gg \left(\frac{\sqrt{N}}{M}\right)^{1/4}.
\end{equation}
If the two largest are, say, $C$ and $D$, then we select
any $\theta_1, t_1$ with $|\theta_{1}|=|t_{1}|^{2}=r \asymp   T_1^2$, and $\theta_2, t_2$ with $|\theta_{2}|=|t_{2}|=s\asymp T_2^2$,
so the number of choices of these are
$$\ll T_1^{2}\log T_1 \cdot   T_2^{2}\log T_2\ll (T_1T_2 \log M)^2.$$
Next we select $\gamma$ with $|\gamma|\asymp C$, and
$\delta$ with $|\delta|\asymp D$, with their arguments in the narrow intervals, of width
$O(M/\sqrt{N})$,  given by the equations \eqref{eq:2arg(gam)+arg(t1)+arg(t2) small} above.
Then, by Lemma \ref{Simple count}, the number of such $\gamma$ is $\ll  1+C^2\frac{M}{\sqrt{N}}$, and the number of such $\delta$ is
$\ll  1+D^2\frac{M}{\sqrt{N}}$. Hence, the ordering assumptions on $C,D,T_{1},T_{2},V$ above, the total number of possible
$(\gamma,\delta,\theta_{1},\theta_{2},t_{1},t_{2},v_{1},v_{2})$ (and hence the corresponding $(\alpha_{j},\beta_{j})$, $j=1,2$)
is
\begin{equation*}
\begin{split}
&\ll (V^{2}\log{M}) \cdot (T_{1}T_{2}\log{M})^{2}\cdot \left(1+C^2\frac{M}{\sqrt{N}} \right)\cdot \left( 1+D^2\frac{M}{\sqrt{N}}\right)
\\&= (CDT_{1}T_{2}V)^{2} (\log{M})^{3} \cdot \left(\frac{1}{C^{2}}+\frac{M}{\sqrt{N}} \right)\cdot \left( \frac{1}{D^{2}}+\frac{M}{\sqrt{N}}\right) \\& \ll    N   \left(\frac{1}{A^2}+ \frac{M}{\sqrt{N}}\right)  \cdot \left(\frac{1}{B^2}+\frac{M}{\sqrt{N}}\right) \cdot
(\log M)^{3},
\end{split}
\end{equation*}
and the analogous expression is proved for every ordering of $C,D,T_1,T_2$.
Thanks to \eqref{eq:AB>>,B>>}, each of these expressions is
$\ll  N^{1/3}M^{4/3} (\log N)^3$, and, since the numbers $C,D,T_{1},T_{2}$, whose product is $\ll N$, are all powers of $2$,
the number of possibilities for
choosing them is $\ll (\log N)^4$, which implies a slightly weaker bound
\begin{equation*}
\sum_{n\leq N} \binom{|\mathcal A(n,M)|}2 \ll N^{1/3}M^{4/3} (\log N)^7
\end{equation*}
as compared to the claimed result \eqref{eq:pairs A << N^1/2M^4/3log^6}.
To improve our analysis in order to match \eqref{eq:pairs A << N^1/2M^4/3log^6} we note that to achieve our bound we fix $AB$ close to
$(\sqrt{N}/M)^{2/3}$, which reduces the number of possibilities for $C,D,T_1,T_2$ to $\ll (\log N)^3$.
\end{proof}

\begin{remark}
There are $\asymp M^2(\log M)^7$ solutions with $C,D,T_1,T_2 \gg (\sqrt{N}/M)^{1/2}$, and we expect this to be the correct number of solutions, provided that $M>N^c$ for some $c>0$.
\end{remark}

\begin{proof}[Proof of Lemma \ref{Simple count}]
Multiplying through by the units we can place $\theta$ in the first quadrant; swapping $x$ and $y$ if necessary we may assume that $0<y<x<N$.
Since the derivative of $\arctan$ is bounded  away from $0$, our set cardinality is bounded by
\begin{equation}
\label{eq:0<y<x<N,y/x-t<eps}
\begin{split}
\#\{ x+iy \in \mathbb Z[i]:\ |x+iy|\leq N, \ &(x,y)=1\ \& \ | \text{arg}(x+iy)-\theta | <\epsilon \}
\\&\ll \# \{ 0<y<x<N:\ |y/x -\theta | <\epsilon \} .
\end{split}
\end{equation}

If there is no solution to \eqref{eq:0<y<x<N,y/x-t<eps} we are done, otherwise we choose one,
i.e. a tuple $(a,b)\in \Z^{2}$ with $0<b<a<N$ and $(a,b)=1$ such that $|b/a-\theta|< \epsilon$. Any other solution
$(x,y)$ of \eqref{eq:0<y<x<N,y/x-t<eps} satisfies
\[
\frac 1{N^2} < \frac 1 {xb} \leq \left|  \frac yx - \frac ab \right| \leq \left|  \frac yx - \theta \right| + \left| \theta - \frac ab \right| <2\epsilon ,
\]
hence there is at most one solution unless $\epsilon>1/2N^2$. The above implies that
every other solution $(x,y)$ of \eqref{eq:0<y<x<N,y/x-t<eps} necessarily satisfies
\begin{equation}
\label{eq:|ax-by|<2eps*bx}
|ax-by|<2\epsilon bx.
\end{equation}

Now, since $(a,b)=1$, there exists an integer solution $(u,v)=(u_{0},v_{0})$ to
$au-bv=1$, and hence for $d\in \Z$ the tuple $(x,y)=(du_{0},dv_{0})$ is one solution to
\begin{equation}
\label{eq:au-bv=d}
ax-by=d.
\end{equation}
Hence the general solution to \eqref{eq:au-bv=d} is given by $x=du_{0}+kb$, $y=dv_{0}+ka$ with $k$ integer;
a solution of \eqref{eq:|ax-by|<2eps*bx} is a solution to \eqref{eq:au-bv=d} with $|d|<2\epsilon bx\ll \epsilon b N$,
and, given a value of $d$, the condition $0<y=dv_{0}+ka <N$ forces $k$ to lie in an interval of length $O(N/a)$.
Therefore the total number of solutions to \eqref{eq:|ax-by|<2eps*bx} is
\begin{equation*}
\ll  \epsilon bN \cdot  \frac{N}{a}\ll \epsilon N^2,
\end{equation*}
which  implies the statement of Lemma \ref{Simple count}.
\end{proof}

\section{Best possible?}

 \subsection{Squarefree $n$: Best possible?}

It is believed that there are lots of primes of the form $a^2+1$. Say $p_1,...,p_k$ are primes,  close to $y^2$, with $p_j=a_j^2+1$
so each $a_j$ is very close to $y$.

Then $P_j = a_j+i$, and these are
each complex numbers with
$|\arg(a_j\pm i)| \ll 1/y$.   Therefore any product $(a_1\pm i)\cdots.(a_k\pm i)$
has argument $\ll k/y$.
Therefore     these points appear in four arcs of width $\ll k/y$
centered around $1,i,-1,-i$.
Here $n = p_1...p_k \approx y^{2k}$ and so the width of these arcs is $\ll  k / n^{1/2k}$.
This yields examples of $n$ with $n$ is squarefree and $r_2(n)\to \infty$ for which
 \[
 \#\{ \alpha, \beta\in \Ec_{n}:\ |\alpha-\beta| \leq n^{1/2-o(1)} \} \gg   |\Ec_{n}|^2  ,
 \]
where $o(1)$ is going to $0$ arbitrarily slowly.

\subsection{Very short arcs}  Fix integer $a$ composed of many prime factors $\equiv 1 \pmod 4$, and select $m$ arbitrarily large, for which
$b=m^2+1$ has $\leq 3$ prime factors.\footnote{We know such $b$ exist by sieve methods; though we believe there are infinitely many such primes, which is as yet unproved.}  Let $u=1$. If  $|\alpha|^2=a$ then $A^*(\alpha,m+i,1)\subset S(ab,2a^{1/2})$, which embeds into $D(ab,2/b^{1/2})\subset D(ab,2/m)$.
Now $|A^*(\alpha,m+i,1)|\gg r_2(a)\gg |D(ab)|$.  Therefore, for $n=ab$, we have given infinitely many examples with
\[
 \# \{ (u,v)\in \Ec_{n}: |u-v|\leq n^{o(1)}\} \gg |\Ec_{n}|.
\]


\end{document}